\documentclass[11pt, letterpaper]{article}

\usepackage{fullpage}
\usepackage{amsmath,amssymb,amsfonts,amsthm,mathtools}
\usepackage[shortlabels]{enumitem}

\usepackage{tikz}
\usetikzlibrary{arrows.meta}
\usepackage{microtype}

\usepackage{thmtools}
\usepackage[hidelinks]{hyperref}
\usepackage[nameinlink,capitalise,noabbrev,nosort]{cleveref}
\crefformat{equation}{(#2#1#3)}
\crefmultiformat{equation}
  {(#2#1#3)}
  { and~(#2#1#3)}
  {,~(#2#1#3)}
  {, and~(#2#1#3)}
\crefrangeformat{equation}
  {(#3#1#4)--(#5#2#6)}

\theoremstyle{definition}
\newtheorem{definition}{Definition}
\newtheorem{observation}[definition]{Observation}

\theoremstyle{plain}
\newtheorem{theorem}[definition]{Theorem}
\newtheorem{proposition}[definition]{Proposition}
\newtheorem{lemma}[definition]{Lemma}
\newtheorem{claim}[definition]{Claim}

\crefname{theorem}{Theorem}{Theorems}
\Crefname{theorem}{Theorem}{Theorems}
\crefname{lemma}{Lemma}{Lemmas}
\Crefname{lemma}{Lemma}{Lemmas}
\crefname{proposition}{Proposition}{Propositions}
\Crefname{proposition}{Proposition}{Propositions}
\crefname{corollary}{Corollary}{Corollaries}
\Crefname{corollary}{Corollary}{Corollaries}
\crefname{definition}{Definition}{Definitions}
\Crefname{definition}{Definition}{Definitions}
\crefname{observation}{Observation}{Observations}
\crefname{appendix}{Appendix}{Appendices}
\Crefname{appendix}{Appendix}{Appendices}
\crefname{claim}{Claim}{Claims}
\Crefname{claim}{Claim}{Claims}

\def \sm {\setminus}
\def \ce {\coloneqq}

\renewcommand{\le}{\leqslant}
\renewcommand{\ge}{\geqslant}

\def \eps {\varepsilon}
\def \es {\varnothing}
\renewcommand \b[2] {\binom{#1}{#2}}

\newcommand{\Rs}{\mathbb R}
\def \nd {[n]^d}
\def \mnd {M(\nd)}
\def \mn {M([n])}
\newcommand{\SF}{\operatorname{SF}}
\def \sfnd {\SF (\nd)}
\newcommand{\supp}{\operatorname{supp}}

\newcommand{\T}{\mathcal T}
\renewcommand{\P}{\mathcal P}

\renewcommand{\o}{\omega{}}
\newcommand{\Zs}{\mathbb Z}

\renewcommand{\S}{\mathcal{S}}

\renewcommand{\Pr}{\mathbb P}
\renewcommand{\l}{\ell}
\newcommand{\G}{\mathcal G}
\newcommand{\B}{\mathcal B}
\newcommand{\R}{\mathcal R}
\newcommand{\U}{\mathcal U}

\newcommand{\Z}{\mathcal Z}

\newcommand{\F}{\mathcal F}

\title{The number of sum-free subsets of lattice cubes}
\author{Haoran Luo\thanks{Department of Mathematics, Statistics and Computer Science, University of Illinois Chicago, Chicago, Illinois, USA. Email: \texttt{haoranl8@uic.edu}. The author was partially supported by an AMS-Simons Travel Grant.}}
\date{}

\begin{document}
\maketitle

\begin{abstract}
A subset of the $d$-dimensional lattice cube $[n]^d$ is sum-free if it contains no solution to the equation $x+y=z$. We study the total number of such subsets. For $d=1$, Cameron and Erd\H{o}s conjectured that the number of sum-free subsets of $[n]$ is $O(2^{n/2})$, and this was proved independently by Green and Sapozhenko. A recent work by Ghosal solved the case $d = 2$. In this paper, we consider all remaining dimensions and prove that for every fixed integer $d \ge 3$, the number of sum-free subsets of $[n]^d$ is $2^{M([n]^d) + O_d(n^{d-1})}$, where $M([n]^d)$ is the maximum possible size of a sum-free subset of $[n]^d$. This verifies a conjecture of Elsholtz and Rackham. Our proof combines the dual weights constructed by Keevash and Lim in their work for $M([n]^d)$, a one-dimensional counting estimate due to Ghosal, a bipartite swapping lemma of Zhao, and a strong fractional entropy inequality of Madiman and Tetali, and it avoids the use of the container lemma or deriving a stability theorem first.
\end{abstract}

\section{Introduction}
Write $[a,b] \ce \{x \in \Zs: a \le x \le b\}$ and $[n] \ce [1,n]$. Let $\nd$ be the $d$-dimensional lattice cube. A set $S \subseteq \nd$ is \emph{sum-free} if there are no $x,y,z \in S$ such that $x+y=z$, where the addition is coordinate-wise and $x, y$ are allowed to be equal. The study of sum-free subsets in various objects has a long history and has attracted considerable attention. We refer the reader to~\cite{cameron1999notes,tao2017sumfree,bedert2025large} and the references within them for more related results.

Let $\mnd$ be the maximum possible size of a sum-free subset of $\nd$. For $d = 1$, it is folklore that $\mn = \lceil n/2 \rceil$, which is achieved by the set of odd integers in $[n]$ and by the interval $[\lfloor n/2\rfloor+1,n]$. For $d \ge 2$, the problem of determining $\mnd$ was asked by Aydinian (for $d= 2$, see~\cite{cameron2005research}) and Cameron~\cite{cameron2002sumfree}; it is also in the list of 100 open problems by Green~\cite{green2018open}. The case $d=2$ was solved by Elsholtz and Rackham~\cite{elsholtz2017maximal}. Later, Lepsveridze and Sun~\cite{lepsveridze2026size} settled the cases for $d = 3,4$, and building on their method, Keevash and Lim~\cite{keevash2026largest} determined the maximum possible density of a sum-free set for all $d$. More precisely, choose $S$ to be the set of the form
\begin{equation} \label{equ::stripe}
\{v \in \nd : u \le v_1 + \cdots + v_d < 2u\}
\end{equation}
with maximum size over all $u$. The results mentioned above showed that, for every dimension $d \ge 2$, we have
\[
\mnd=|S|+O_d(n^{d-1}).
\]
For example, for $d=2$, one can take $u = \lfloor 4n/5 \rfloor$ and this stripe has asymptotic size $0.6n^2$. We note that several stability results were also obtained. For $d=1$, the results of Freiman~\cite{freiman1992structure}, Deshouillers, Freiman, S{\'o}s, and Temkin~\cite{deshouillers1999structure}, and Tran~\cite{tran2018structure} imply that every sum-free set $S \subseteq [n]$ with $|S|=(1/2-o(1))n$ has symmetric difference $o(n)$ from either the set of odd integers in $[n]$ or the interval $[\lfloor n/2\rfloor+1,n]$. For $d=2$, Liu, Wang, Wilkes, and Yang~\cite{liu2023shape} proved that if a sum-free subset of $[n]^2$ has size close to $M([n]^2)$, then itself must be close to the optimal stripe in \cref{equ::stripe}. For $d \ge 3$, the corresponding stability problem remains open.

Having settled the asymptotic problem for $\mnd$, the next natural problem is to upgrade these results to a counting result. Let $\sfnd$ be the total number of sum-free subsets of $\nd$. Note that every subset of a sum-free set is still sum-free, so trivially
\[
\sfnd \ge 2^{\mnd}.
\]
A very natural question is how far this bound is from the truth. For $d=1$, a well-known conjecture by Cameron and Erd\H{o}s~\cite{cameron1990number} states that the number of sum-free subsets of $[n]$ is $O(2^{n/2})$. After a series of works by Cameron and Erd\H{o}s~\cite{cameron1990number}, Calkin~\cite{calkin1990number}, Alon~\cite{alon1991independent}, Erd\H{o}s and Granville (unpublished, see~\cite{green2004cameron}), Freiman~\cite{freiman1992structure}, and Omel'yanov and Sapozhenko~\cite{omelyanov2002number}, Green~\cite{green2004cameron} and Sapozhenko~\cite{sapozhenko2008cameron} finally proved the full Cameron--Erd\H{o}s conjecture independently. In modern terminology, both proofs construct suitable families of containers, Green using Fourier analysis and Sapozhenko using a purely combinatorial method, and then reduce the count to subsets of the odd integers and sum-free subsets of $[\lceil(n+1)/3\rceil,n]$, the latter having already been solved by Cameron and Erd\H{o}s~\cite{cameron1990number}.

For general $d$, Elsholtz and Rackham~\cite{elsholtz2017maximal} made the analogous conjecture\footnote{Strictly speaking, Elsholtz and Rackham stated their conjecture using slightly different notation. Here we use an equivalent formulation in terms of $\mnd$, given the result of $\mnd$ by Keevash and Lim~\cite{keevash2026largest}.} that, for every $d$, $\sfnd=2^{\mnd+O_d(n^{d-1})}$. Ghosal~\cite{ghosal2025number} recently proved the conjecture for $d = 2$ and obtained a weaker estimate $\sfnd=2^{\mnd+o(n^d)}$ for every $d \ge 3$ by using Green's arithmetic removal lemma~\cite{green2005regularity} and the hypergraph container method~\cite{balogh2015independent,saxton2015hypergraph}; for $d = 2$, Ghosal additionally used the stability theorem of Liu, Wang, Wilkes, and Yang~\cite{liu2023shape} mentioned above.

In this paper, we settle all the remaining dimensions.

\begin{theorem} \label{thm::main}
For every integer $d \ge 3$, we have
\[
\sfnd = 2^{\mnd + O_d (n^{d-1})}.
\]
\end{theorem}

Our proof associates every sum-free subset with a \emph{profile} (defined in \cref{sec::profile}) and then reduces \cref{thm::main} to two counting propositions. The first proposition gives an upper bound on the number of sum-free sets with a given profile. The second proposition bounds the number of profiles. Combining the two propositions gives \cref{thm::main}. We give more details in \cref{sec::profile} after we introduce the necessary notation. We remark that, different from previous works, our arguments count the sum-free sets directly and avoid the use of a container lemma and the need to derive a stability theorem first.

We remark that the counting problem for sum-free subsets has also been studied in other settings, see for example~\cite{alon1991independent, green2005sumfree, alon2014counting, alon2014refinement, balogh2015number}.

The rest of this paper is organized as follows. In \cref{sec::profile}, we first recall the fiber linear program and the dual weights used by Keevash and Lim in their work on $\mnd$. Then we define the profiles of sum-free subsets and defects of profiles, and reduce \cref{thm::main} to two counting propositions. In \cref{sec::keevashlim}, we record the precise properties of the Keevash--Lim weights needed in our proof. In \cref{sec::fixedProfileProof,sec::profileCountProof}, we prove these two counting propositions, the fixed-profile estimate and the profile-count estimate, respectively. In \cref{sec::concluding}, we give some concluding remarks, including other methods we tried for this work and some potential open problems.

\section{Fibers and two counting propositions} \label{sec::profile}
In this section, we introduce the objects needed for the proof of \cref{thm::main} and reduce it to two counting propositions. We first introduce fibers and the five regions appearing in the Keevash--Lim weights. Then, we define the height profiles and their weighted defects, and state the two propositions that lead to the main theorem.

Hereinafter, for real numbers $a$ and $b$, we write $[a,b]_{\Rs} \ce \{x \in \Rs : a \le x \le b\}$. Similarly, $(a,b]_{\Rs}$, $[a,b)_{\Rs}$, and $(a,b)_{\Rs}$ denote the corresponding real intervals.

\subsection{Fibers and regions} \label{subsec::fiberAndRegions}
We write $Q \ce [0,n]^d$. Let $M(Q)$ and $\SF(Q)$ denote the maximum possible size and the total number of sum-free subsets of $Q$, respectively. Instead of working directly on $\nd$, we first prove the corresponding counting estimate for $Q$. The reason is that the linear program and the dual weights of Keevash and Lim~\cite{keevash2026largest} are constructed for $Q$. At the end of this section, we will transfer the result back to $\nd$, where we simply use the fact that $Q \sm \nd$ has $O_d(n^{d-1})$ points.

Let
\[
V \ce [0,n]^{d-1}.
\]
We view every element of $V$ as the first $d-1$ coordinates in $Q$. For every $v \in V$, define the vertical \emph{fiber} above $v$ by\footnote{For a vector $v$ and a number $t$, we write $(v,t)$ for their concatenation.} $F_v \ce \{(v,t) : t \in [0,n]\}$. For a set $S \subseteq Q$, let $S_v \ce \{t \in [0,n] : (v,t) \in S\}$. Hence, $S \cap F_v=\{(v,t) : t \in S_v\}$.

The connection between $V$ and sum-free subsets is based on the following simple observation of Lepsveridze and Sun; see Lemma~2.3 in~\cite{lepsveridze2026size}.

\begin{observation} \label{obs::threefiber}
For every sum-free set $S \subseteq Q$ and every $x,y,z \in V$ with $x+y=z$, we have
\[
|S_x|+|S_y|+|S_z| \le 2(n+1).
\]
\end{observation}

\begin{proof}
The result is immediate if $S_z=\es$. Otherwise, let $c \ce \max S_z$, and hence $|S_z| \le c+1$. For every integer $t \in [0,c]$, the points $(x,t)$ and $(y,c-t)$ cannot both belong to $S$, since their sum is $(z,c) \in S$. Summing over $t$ gives
\[
|S_x \cap [0,c]|+|S_y \cap [0,c]| \le c+1.
\]
Then,
\begin{align*}
    |S_x|+|S_y|+|S_z| & = |S_x \cap [0,c]|+ |S_y \cap [0,c]| +
    |S_x \sm [0,c]|+|S_y \sm [0,c]| + |S_z| \\
    & \le (c+1) + (n+1 - (c+1)) + (n+1 - (c+1)) + (c+1) \le 2(n+1). \qedhere
\end{align*}
\end{proof}

For sets $X,Y,Z \subseteq V$, let $\T(X,Y,Z) \ce \{(x,y,z) \in X \times Y \times Z : x+y=z\}$; we write $\T$ for $\T(V,V,V)$. For every $(x,y,z) \in \T$, we say that $x$ and $y$ are the \emph{sources} and $z$ is the \emph{target}; we also say that $F_x$ and $F_y$ are the source fibers and $F_z$ is the target fiber. Whenever a function defined on $\T$ is summed over triples $(x,y,z)\in V^3$, we extend it by zero to all of $V^3$. By \cref{obs::threefiber}, on every triple in $\T$, at least one third of the total capacity of the three fibers is missing. To make the best use of this, the following linear program is considered. For a sum-free set $S \subseteq Q$, let $a_v \ce |S_v|$ for every $v \in V$. Then, $(a_v)_{v \in V}$ is feasible for the following linear program, whose optimum is an upper bound for $|S|$:
\[
\begin{array}{ll}
 \text{maximize}&\displaystyle\sum_{v \in V}a_v,\\[1mm]
 \text{subject to}&a_x+a_y+a_z \le 2(n+1)
       \quad((x,y,z) \in \T),\\
 &0 \le a_v \le n+1\quad(v \in V).
\end{array}
\tag{Primal}\label{equ::primalLP}
\]

In order to give an upper bound for the optimum of \cref{equ::primalLP}, we can consider its dual program. For every function $\o:\T \to [0,\infty)_{\Rs}$, let
\begin{equation} \label{equ::defW}
W_\o(v) \ce \sum_{\substack{y,z \in V\\v+y=z}}\o(v,y,z) +\sum_{\substack{x,z \in V\\x+v=z}}\o(x,v,z) +\sum_{\substack{x,y \in V\\x+y=v}}\o(x,y,v).
\end{equation}
Note that the coordinates are counted with multiplicity; for example, $(v,v,2v)$ contributes twice to $W_\o(v)$, while $(0,0,0)$ contributes three times to $W_\o(0)$. The dual program of \cref{equ::primalLP} is
\[
\begin{array}{ll}
 \text{minimize}&\displaystyle
   2(n+1)\sum_{(x,y,z) \in \T}\o(x,y,z)
   +(n+1)\sum_{v \in V}\beta(v),\\[1mm]
 \text{subject to}&\displaystyle
   W_\o(v)+\beta(v) \ge 1\quad(v \in V),\\
 &\o(x,y,z) \ge 0\quad((x,y,z) \in \T),\\
 &\beta(v) \ge 0\quad(v \in V).
\end{array}
\tag{Dual}\label{equ::dualLP}
\]
Every feasible solution to \cref{equ::dualLP} gives an upper bound on $|S|$. Lepsveridze and Sun~\cite{lepsveridze2026size} used this linear program to settle the cases $d=3,4$. Building on their method, Keevash and Lim~\cite{keevash2026largest} constructed dual solutions with the correct asymptotic value for every fixed $d \ge 3$. We introduce the five regions used in their constructions here and record the precise properties needed in our proof in \cref{sec::keevashlim}.

For every vector $v=(v_1,\ldots,v_j)$, let $h(v) \ce v_1+\ldots+v_j$. Let $u_d^*$ be the unique value of $u \ge 0$ maximizing the volume of $\{v \in [0,1]_{\Rs}^d : u \le h(v)<2u\}$, and let $m \ce \lfloor u_d^*n\rfloor$. Keevash and Lim showed in Sections~1 and~3 of~\cite{keevash2026largest} that $d/3<u_d^*<d/2$ for every $d \ge 2$. For example, $u_3^* = (15-\sqrt{15}) / 10 \approx 1.11$ and $u_4^* \approx 1.45$. For all sufficiently large $n$, the set $V$ is partitioned into the following five regions:
\begin{align*}
A \ce \{v : 0 \le h&(v) < m-n\}, \,\,
B \ce \{v : m-n \le h(v) \le m\}, \,\,
C \ce \{v : m < h(v) < 2m-n\},\\
D& \ce \{v : 2m-n \le h(v) \le 2m\}, \quad
E \ce \{v : 2m < h(v) \le (d-1)n\}.
\end{align*}
We note that when $d=3$, the region $E$ is empty. We further divide $C$ in the middle:\footnote{Keevash and Lim~\cite{keevash2026largest} also divide $C$ into two regions in their proof. We remark that our division $C = C^- \cup C^+$ is different.}
\[
C^- \ce \left\{v \in C : h(v) \le \left\lfloor {(3m-n-1)}/{2}\right\rfloor\right\}, \qquad C^+ \ce C \sm C^-.
\]
See \cref{fig::baseHeightRegions} for a schematic picture of the regions for $d=3$.

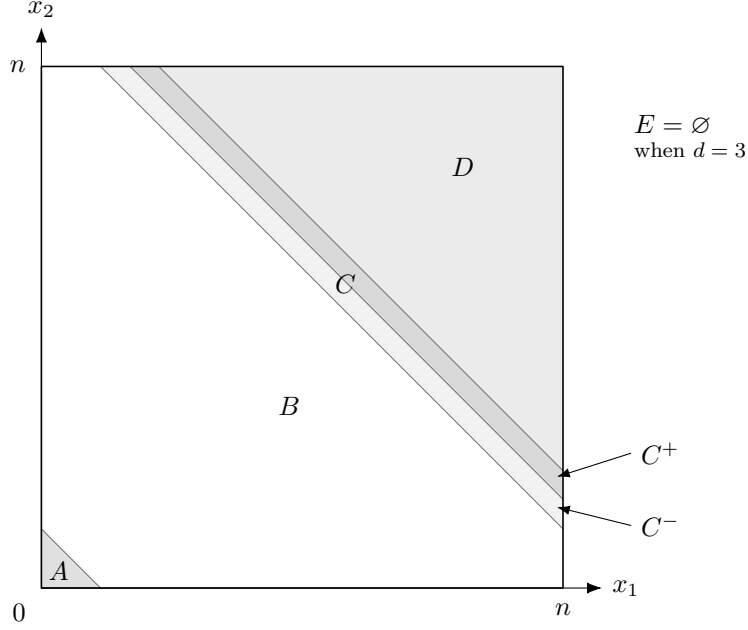
\begin{figure}[ht]
\centering
\begin{tikzpicture}[
  x=1.15cm,
  y=1.15cm,
  line cap=round,
  line join=round,
  every node/.style={font=\small},
  boundary/.style={black!55}
]
  \def\ab{0.68}
  \def\cmid{1.02}
  \def\cd{1.35}

  \path[fill=black!12]
    (0,0)--(\ab,0)--(0,\ab)--cycle;
  \path[fill=white]
    (\ab,0)--(6,0)--(6,\ab)--(\ab,6)--(0,6)--(0,\ab)--cycle;
  \path[fill=black!5]
    (\ab,6)--(\cmid,6)--(6,\cmid)--(6,\ab)--cycle;
  \path[fill=black!15]
    (\cmid,6)--(\cd,6)--(6,\cd)--(6,\cmid)--cycle;
  \path[fill=black!8]
    (\cd,6)--(6,6)--(6,\cd)--cycle;

  \draw[boundary] (0,\ab)--(\ab,0);
  \draw[boundary] (\ab,6)--(6,\ab);
  \draw[boundary] (\cmid,6)--(6,\cmid);
  \draw[boundary] (\cd,6)--(6,\cd);
  \draw[semithick] (0,0) rectangle (6,6);

  \draw[-{Latex[length=2mm]}] (0,0)--(6.45,0)
    node[right] {$x_1$};
  \draw[-{Latex[length=2mm]}] (0,0)--(0,6.45)
    node[above] {$x_2$};
  \node[below left=2pt] at (0,0) {$0$};
  \node[below=2pt] at (6,0) {$n$};
  \node[left=2pt] at (0,6) {$n$};

  \node at (0.20,0.20) {$A$};
  \node at (2.85,2.10) {$B$};
  \node at (3.50,3.50) {$C$};
  \node at (4.85,4.85) {$D$};

  \node[anchor=west] (cplus) at (6.78,1.55) {$C^+$};
  \draw[-{Latex[length=1.7mm]},thin]
    (cplus.west)--(5.93,1.28);
  \node[anchor=west] (cminus) at (6.78,0.72) {$C^-$};
  \draw[-{Latex[length=1.7mm]},thin]
    (cminus.west)--(5.93,0.93);

  \node[anchor=west,align=left] at (6.70,5.20)
    {$E=\es$\\[-1mm]\scriptsize when $d=3$};
\end{tikzpicture}
\caption{The regions in the two-dimensional base $V$ when $d=3$.}
\label{fig::baseHeightRegions}
\end{figure}

Keevash and Lim~\cite{keevash2026largest} construct a nonnegative weight function $\o:\T \to [0,\infty)_{\Rs}$ of the form $\o=\o_1+\o_2+\o_3$, where $\o_1,\o_2,\o_3$ are supported on $\T(B,B,D)$, $\T(A,C,C)$, and $\T(C,C,E)$, respectively. Let $W_\o(v)$ be the function defined as in \cref{equ::defW} and $\beta(v) \ce \max(1-W_\o(v), 0)$. Then, $(\o,\beta)$ is a feasible solution to the dual program in \cref{equ::dualLP} and gives the required asymptotic upper bound $M(Q) + O_d (n^{d-1})$ for the primal program.

For our counting argument, we need slightly more information from their construction. The reason is that the region $C$ plays different roles in $\o_2$ and $\o_3$, and this requires us to use different information about $C$. Therefore, we use exactly the same construction but use the division $C=C^- \cup C^+$ above and refine the supports accordingly. We will also use bounded translations of the first two weight components. These refinements and the precise properties needed in our proof are given in \cref{sec::keevashlim}.

\subsection{Profile and defect} \label{subsec::profileanddefect}

\begin{definition}[Height profile] \label{def::heightProfile}
For a set $S \subseteq Q$, its \emph{height profile} is the function $\P(S):V \to [-1,n+1]$ defined by
\[
[\P(S)](v) \ce
\begin{cases}
  \min\bigl(S_v \cup \{n+1\}\bigr),&v \in A \cup B \cup C^-,\\[1mm]
  \max\bigl(S_v \cup \{-1\}\bigr),&v \in C^+ \cup D \cup E.
\end{cases}.
\]
More generally, a \emph{profile} is any function $\P:V \to [-1,n+1]$ satisfying $\P(v) \in [0,n+1]$ for $v \in A \cup B \cup C^-$ and $\P(v) \in [-1,n]$ for $v \in C^+ \cup D \cup E$.
\end{definition}
\noindent We call the restriction of a profile to $A \cup B \cup C^-$ its \emph{lower} part and its restriction to $C^+ \cup D \cup E$ its \emph{upper} part. The values $n+1$ and $-1$ are just to record empty fibers.

Let $e_1 \ce (1,0,\ldots,0) \in \Zs^{d-1}$. Let
\begin{equation} \label{equ::shiftSet}
 \R \ce \{(0,0),(e_1,0),(0,e_1),(2e_1,0),(0,2e_1),(e_1,e_1)\};
\end{equation}
here $0$ denotes the zero vector in $\Zs^{d-1}$.

Let $\o_1,\o_2,\o_3:\T \to [0,\infty)_{\Rs}$ be the three weights whose existence and precise properties will be stated in \cref{lem::KLWeights} in \cref{subsec::threeWeights}, where we shall have $\supp(\o_1) \subseteq \T(B,B,D)$, $\supp(\o_2) \subseteq \T(A,C^-,C^+)$, and $\supp(\o_3) \subseteq \T(C^-,C^-,E)$. As explained in \cref{subsec::fiberAndRegions}, they are exactly the three components of the dual weight constructed by Keevash and Lim~\cite{keevash2026largest}; our division $C=C^- \cup C^+$ only refines their supports. We postpone their formal statement and first introduce the main idea of our proof.

\begin{definition}[Local defects] \label{def::localDefect}
Fix a profile $\P$. For $(x,y,z) \in \T(B,B,D)$ and $(r,s) \in \R$ such that $(x+r,y+s,z+r+s) \in \T(B,B,D)$, define
\[
\delta_1^{r,s}(x,y,z) \ce \P(z+r+s)-\P(x+r)-\P(y+s).
\]
For $(x,y,z) \in \T(A,C^-,C^+)$ and $(r,s) \in \R$ such that $(x+r,y+s,z+r+s) \in \T(A,C^-,C^+)$, define
\[
\delta_2^{r,s}(x,y,z) \ce \P(z+r+s)-\P(x+r)-\P(y+s).
\]
For $(x,y,z) \in \T(C^-,C^-,E)$, define
\[
\delta_3(x,y,z) \ce \P(z)-\P(x)-\P(y).
\]
\end{definition}

With slight abuse of notation, we also use $\delta$ for the following global defects.
\begin{definition}[Weighted defects] \label{def::defectSums}
For a profile $\P$ and $(r,s) \in \R$, let
\begin{align*}
 \delta_1^{r,s}(\P)
 & \ce \sum_{\substack{(x,y,z) \in \T(B,B,D)\\
          (x+r,y+s,z+r+s) \in \T(B,B,D)}}
       \o_1(x,y,z)|\delta_1^{r,s}(x,y,z)|,\\
 \delta_2^{r,s}(\P)
 & \ce \sum_{\substack{(x,y,z) \in \T(A,C^-,C^+)\\
          (x+r,y+s,z+r+s) \in \T(A,C^-,C^+)}}
       \o_2(x,y,z)|\delta_2^{r,s}(x,y,z)|,\\
\delta_3(\P) & \ce \sum_{(x,y,z) \in \T(C^-,C^-,E)} \o_3(x,y,z)|\delta_3(x,y,z)|,
\end{align*}
and define
\begin{equation}
 \delta(\P) \ce
 \max\Bigl\{
   \max_{(r,s) \in \R}\delta_1^{r,s}(\P),\,\,
   \max_{(r,s) \in \R}\delta_2^{r,s}(\P),\,\,
   \delta_3(\P)
 \Bigr\}.
\end{equation}
\end{definition}

The local defects are signed, while their absolute values measure the failure of the corresponding profile values to satisfy the additive relation. The unshifted defects already give the local exponential savings needed for the fixed-profile estimate (see \cref{prop::localEntropy}). The translated defects are included because four nearby additive relations isolate a second difference of the profile, which is used to count the possible profiles in \cref{sec::profileCountProof}.

There is a reason for separating the fixed-profile count from the profile count. If one counts each triple of fibers without fixing its profile values, the local count contains an unavoidable polynomial factor in $n$. In the case $d=2$, Ghosal showed that the unrestricted count on one triple of fibers is $\Theta(n^2 2^{2n})$, which gives an error term $O(n\log n)$ in the exponent; see the end of Section~3.3 in~\cite{ghosal2025number}. In our weighted setting, the same local loss would contribute $O_d(n^{d-1}\log n)$ to the exponent. Fixing the profile removes this polynomial loss from the local estimate, while the weighted profile count controls all profiles globally, with the defect compensating for the number of possible profiles and leaving only an $O_d(n^{d-1})$ error.

Our use of translated defects is partially inspired by Ghosal's proof for $d=2$; see Section~3.5 in~\cite{ghosal2025number}. Ghosal uses shifted defects to control the second differences of a one-dimensional height profile and then reconstructs the profile from its first two values and its second differences. Our argument is to make this cancellation-and-reconstruction mechanism compatible with the Keevash--Lim dual weights in every fixed dimension. We introduce this weighted defect and the fixed collection of bounded translations whose signed cancellations isolate second differences along lines parallel to $e_1$, while the losses near the boundaries of the regions remain negligible. The same weighted defect, after translating the relevant dual component, also gives the exponential saving in the fixed-profile estimate. Thus, a single quantity controls both the number of sum-free sets with a given profile and the number of possible profiles.

\subsection{Two main propositions} \label{subsec::twoprop}
For a profile $\P$, let $N(\P)$ denote the number of sum-free subsets $S \subseteq Q$ with $\P(S)=\P$. The proof of \cref{thm::main} reduces to the following two statements. All implicit constants are uniform over the profiles.

\begin{proposition}[Fixed-profile count] \label{prop::fixedProfile}
There is an absolute constant $\alpha_0>0$ such that, for every fixed integer $d \ge 3$, every sufficiently large even integer $n$, and every profile $\P$,
\[
N(\P) \le 2^{M(Q)-\alpha_0\delta(\P)+O_d(n^{d-1})}.
\]
\end{proposition}

\begin{proposition}[Profile count] \label{prop::profileCount}
For every fixed integer $d \ge 3$, every $\eps>0$, every sufficiently large even integer $n$, and every real number $R \ge 0$, we have
\[
\bigl|\{\P : \delta(\P) \le R\}\bigr|
 \le 2^{\eps R+O_{d,\eps}(n^{d-1})}.
\]
The implicit constant is independent of $R$.
\end{proposition}
\noindent We note that \cref{prop::profileCount} is slightly stronger than what we need: it counts all profiles with defect at most $R$, not just those realized by sum-free subsets of $Q$. We also remark that the assumption that $n$ is even is made only so that we can use the Keevash--Lim dual weights in the form recorded in \cref{sec::keevashlim}. Once these weight properties are available, the counting arguments do not use the parity of $n$.

\cref{thm::main} follows from these two propositions immediately.
\begin{proof}[Proof of \cref{thm::main} assuming \cref{prop::fixedProfile,prop::profileCount}]
We first prove
\begin{equation} \label{equ::zeroBasedGoal}
 \SF(Q) \le 2^{M(Q)+O_d(n^{d-1})}.
\end{equation}
Assume first that $n$ is even and sufficiently large. For every integer $r \ge 0$, applying \cref{prop::profileCount} with $\eps=\alpha_0/2$ and $R=r+1$, we obtain
\[
\bigl|\{\P : \delta(\P)<r+1\}\bigr|
 \le
2^{\frac{\alpha_0}{2}(r+1)+O_d(n^{d-1})}.
\]
Therefore, by \cref{prop::fixedProfile},
\begin{align*}
\SF(Q)
&=\sum_{r \ge 0}
  \sum_{\substack{\P\\r \le \delta(\P)<r+1}}N(\P) \le
\sum_{r \ge 0}
2^{M(Q)-\alpha_0r+O_d(n^{d-1})}
\cdot \bigl|\{\P : \delta(\P)<r+1\}\bigr|\\
& \le
2^{M(Q)+O_d(n^{d-1})}
\sum_{r \ge 0}2^{-\alpha_0r/2}
=2^{M(Q)+O_d(n^{d-1})} \cdot O(1)
=2^{M(Q)+O_d(n^{d-1})}.
\end{align*}
If $n$ is odd, let $Q^+ \ce [0,n+1]^d$. Then $Q \subseteq Q^+$, and the even case applied with $n+1$ gives
\[
\SF(Q) \le \SF(Q^+) \le 2^{M(Q^+)+O_d(n^{d-1})} \le 2^{M(Q)+O_d(n^{d-1})},
\]
because $M(Q^+) - M(Q) \le |Q^+ \sm Q|=O_d(n^{d-1})$.

Finally, $\nd \subseteq Q$ and $|Q \sm \nd|=(n+1)^d-n^d=O_d(n^{d-1})$. Hence, $M(Q) \le \mnd+O_d(n^{d-1})$ and $\sfnd \le \SF(Q)$. Together with \cref{equ::zeroBasedGoal} and the trivial lower bound $\sfnd \ge 2^{\mnd}$, this proves \cref{thm::main}.
\end{proof}

\section{The dual weights of Keevash and Lim} \label{sec::keevashlim}
In this section, we record the properties of the dual weights constructed by Keevash and Lim~\cite{keevash2026largest} that are needed in our proof. We first give their support properties, the estimates for the dual covering function, and the coordinate-sum formulas. Then, we show that these properties are stable under the translations used in the definition of $\delta(\P)$ and deduce the required bound on the dual objective.

Throughout this section, let $d \ge 3$ be fixed. All statements hold for every sufficiently large even integer $n$, and all implicit constants depend only on $d$.

\subsection{The three weights and their coordinate sums} \label{subsec::threeWeights}
Recall that in \cref{subsec::fiberAndRegions}, for every function $\o:\T \to [0,\infty)_{\Rs}$, we define
\[
W_\o(v) \ce \sum_{\substack{y,z \in V\\v+y=z}}\o(v,y,z) +\sum_{\substack{x,z \in V\\x+v=z}}\o(x,v,z) +\sum_{\substack{x,y \in V\\x+y=v}}\o(x,y,v).
\]

The following properties and the original support statements are given by Lemma~4.3 in~\cite{keevash2026largest}. The refinement of the supports from $C$ to $C^-$ and $C^+$ is verified in \cref{app::KLCoordinateSums}.

\begin{lemma} \label{lem::KLWeights}
There are nonnegative functions\footnote{When $d=3$, we have $\o_3=0$.} $\o_1,\o_2,\o_3:\T \to [0,\infty)_{\Rs}$ with the following properties.
\[
\supp(\o_1) \subseteq \T(B,B,D),\qquad
\supp(\o_2) \subseteq \T(A,C^-,C^+),\qquad
\supp(\o_3) \subseteq \T(C^-,C^-,E).
\]
Let $\o \ce \o_1+\o_2+\o_3$ and $W \ce W_\o$. Then, $W(v) \le 1$ for every $v \in C$ and
\[
\sum_{v \in V \sm C}|W(v)-1|=O_d(n^{d-2}).
\]
\end{lemma}

Throughout the rest of the paper, we fix the functions $\o_1,\o_2,\o_3$ given by \cref{lem::KLWeights}, and let $\o \ce \o_1+\o_2+\o_3$ and $W \ce W_\o$.

We next record the coordinate sums needed later. Their proofs are given in \cref{app::KLCoordinateSums}. These bounds ensure that the constant errors in the local counting estimates contribute only $O_d(n^{d-1})$ to the exponent.

\begin{lemma}[Coordinate sums of $\o_1$] \label{lem::w1Weights}
The function $\o_1$ supplied by \cref{lem::KLWeights} satisfies the following properties. For every $v \in B$, we have
\begin{equation} \label{equ::w1CoordinateSums}
 \sum_{y,z \in V}\o_1(v,y,z)
 =\sum_{x,z \in V}\o_1(x,v,z)
 =\frac{1}{2}+O_d(n^{-1}).
\end{equation}
Moreover, $\sum_{x,y \in V}\o_1(x,y,v)=1$ for every $v \in D$, and $\sum_{x,y,z \in V}\o_1(x,y,z)=O_d(n^{d-1})$.
\end{lemma}

\begin{lemma}[Coordinate sums of $\o_2$] \label{lem::w2Weights}
The function $\o_2$ supplied by \cref{lem::KLWeights} satisfies the following properties. For every $v \in A$, we have
\[
\sum_{y,z \in V}\o_2(v,y,z) =
\begin{cases}
  1,&1 \le h(v) \le m-n-2,\\
  0,&\text{otherwise}.
\end{cases}.
\]
For $v \in C^-$ and $v' \in C^+$, let
\[
\lambda_-(v) \ce \sum_{x,z \in V}\o_2(x,v,z), \qquad \lambda_+(v') \ce \sum_{x,y \in V}\o_2(x,y,v').
\]
The values $\lambda_-(v)$ and $\lambda_+(v')$ depend only on $h(v)$ and $h(v')$, respectively, and both satisfy $0<\lambda_-(v),\lambda_+(v') \le 1$ throughout their respective regions. Moreover,
\begin{equation} \label{equ::lambdaLogSum}
  \sum_{v \in C^-}\log_2\Bigl(2+\frac{1}{\lambda_-(v)}\Bigr)
  +\sum_{v' \in C^+}\log_2\Bigl(2+\frac{1}{\lambda_+(v')}\Bigr)
  =O_d(n^{d-1}),
\end{equation}
and $\sum_{x,y,z \in V}\o_2(x,y,z)=O_d(n^{d-1})$.
\end{lemma}

\begin{lemma}[Coordinate sums of $\o_3$] \label{lem::w3Weights}
The function $\o_3$ supplied by \cref{lem::KLWeights} satisfies the following properties. For every $v \in E$ with $h(v) \ge 2m+2$, we have $\sum_{x,y \in V}\o_3(x,y,v)=1$. For $v$ with $h(v)=2m+1$, this sum is $0$. Moreover, $\sum_{x,y,z \in V}\o_3(x,y,z)=O_d(n^{d-1})$.
\end{lemma}

\subsection{Translated weights} \label{subsec::transWeights}
The definition of $\delta(\P)$ in \cref{subsec::profileanddefect} uses translations of $\o_1$ and $\o_2$ by vectors in $\R$. Here, we show that replacing one weight component by such a translation changes the dual covering function $W$ by only $O_d(n^{d-2})$.

Fix $(r,s) \in \R$. Let
\begin{align*}
\G_1^{r,s} & \ce \bigl\{(x,y,z) \in \supp(\o_1) : (x+r,y+s,z+r+s) \in \T(B,B,D)\bigr\}, \\
\G_2^{r,s} & \ce \bigl\{(x,y,z) \in \supp(\o_2) : (x+r,y+s,z+r+s) \in \T(A,C^-,C^+)\bigr\}.
\end{align*}
For $i \in \{1,2\}$, let $\B_i^{r,s} \ce \supp(\o_i) \sm \G_i^{r,s}$. Thus, $\G_i^{r,s}$ and $\B_i^{r,s}$ are, respectively, the good and bad triples for the shift $(r,s)$.

For $i \in \{1,2\}$, let $\o_i^{r,s}:\T \to [0,\infty)_{\Rs}$ be the function defined by
\[
\o_i^{r,s}(x,y,z) \ce
\begin{cases}
  \o_i(x-r,y-s,z-r-s),
  &\text{if }(x-r,y-s,z-r-s) \in \G_i^{r,s},\\
  0,&\text{otherwise}.
\end{cases}.
\]
For $t \in \Rs$, let $(t)_+ \ce \max\{t,0\}$. The proof of the following lemma is given in \cref{app::KLCoordinateSums}.

\begin{lemma} \label{lem::translationVariation}
Fix $(r,s) \in \R$ and $i \in \{1,2\}$. Let $\o' \ce \o-\o_i+\o_i^{r,s}$ and $W' \ce W_{\o'}$. Then,
\[
\sum_{(x,y,z) \in \B_i^{r,s}}\o_i(x,y,z) =O_d(n^{d-2}),
\]
and
\begin{equation} \label{equ::translationVariation}
 \sum_{v \in V}|W'(v)-W(v)|=O_d(n^{d-2}).
\end{equation}
Consequently,
\begin{equation} \label{equ::translatedWeightError}
 \sum_{v \in C}(W'(v)-1)_+ + \sum_{v \in V \sm C}|W'(v)-1|
 =O_d(n^{d-2}).
\end{equation}
\end{lemma}

For later use, define the common good and bad sets by
\[
\G_i \ce \bigcap_{(r,s) \in \R}\G_i^{r,s},
\qquad
\B_i \ce \supp(\o_i) \sm \G_i
=
\bigcup_{(r,s) \in \R}\B_i^{r,s}
\qquad (i \in \{1,2\}).
\]
Hence, every shift in $\R$ is valid on every triple in $\G_i$. Also, the first assertion of \cref{lem::translationVariation} and a union bound over the six pairs in $\R$ give
\begin{equation} \label{equ::commonValidMass}
\sum_{(x,y,z) \in \B_i}
\o_i(x,y,z)
=O_d(n^{d-2})
\qquad (i \in \{1,2\}).
\end{equation}

\subsection{The dual objective}
We now use the extremal stripe to bound the objective value of every translated certificate. Recall that in \cref{subsec::fiberAndRegions}, we defined $u_d^*$ to be the unique value of $u \ge 0$ maximizing the volume of $\{v \in [0,1]_{\Rs}^d : u \le h(v)<2u\}$ and let $m \ce \lfloor u_d^*n\rfloor$. The following elementary identity is the integer-threshold version of the stripe calculation from Lemma~4.4 in~\cite{keevash2026largest}. Their formulation uses the real threshold $u_d^*n$, while our use of the integer threshold $m$ makes the equality below exact. Let $S_{d,n} \ce \{v \in Q : m \le v_1+\cdots+v_d<2m\}$.

\begin{lemma} \label{lem::stripeFibers}
The set $S_{d,n}$ is sum-free. If $(x,y,z) \in \T(B,B,D) \cup \T(A,C,C) \cup \T(C,C,E)$, then
\begin{equation} \label{equ::stripeTriple}
   |S_{d,n} \cap F_x|+|S_{d,n} \cap F_y|+|S_{d,n} \cap F_z|=2(n+1).
\end{equation}
Moreover, $|S_{d,n} \cap F_v|=n+1$ for every $v \in C$.
\end{lemma}

\begin{proof}
The set $S_{d,n}$ is sum-free by its definition. Its fiber is empty on $A \cup E$ and is all of $[0,n]$ on $C$. If $v \in B$, then $|S_{d,n} \cap F_v|=n+1-m+h(v)$; if $v \in D$, then $|S_{d,n} \cap F_v|=2m-h(v)$. Thus, the three fiber sizes sum to $2(n+1)$ on $B \times B \times D$. On $A \times C \times C$ and $C \times C \times E$, they are respectively $0,n+1,n+1$ and $n+1,n+1,0$, so they also sum to $2(n+1)$.
\end{proof}

For the unchanged weight $\o$, the next estimate is the dual-objective calculation in the proof of Theorem~1.1 in~\cite{keevash2026largest}, rewritten in our notation. The point needed here is that the same estimate remains valid after one of the first two weight components is replaced by a translation.

\begin{lemma} \label{lem::objective}
Let $i \in \{1,2\}$ and $(r,s) \in \R$, and let $\o' \ce \o-\o_i+\o_i^{r,s}$. Let $W' \ce W_{\o'}$ and $\beta(v) \ce (1-W'(v))_+$ for every $v \in V$. Then,
\begin{equation} \label{equ::equinLemObjective}
2(n+1)\sum_{(x,y,z) \in \T}\o'(x,y,z) +(n+1)\sum_{v \in V}\beta(v) \le M(Q)+O_d(n^{d-1}).
\end{equation}
\end{lemma}

\begin{proof}
Let $a(v) \ce |S_{d,n} \cap F_v|$ for every $v \in V$. By \cref{equ::stripeTriple}, we have
\[
2(n+1)\sum_{(x,y,z) \in \T}\o'(x,y,z)
=
\sum_{(x,y,z) \in \T} \o'(x,y,z) \big(a(x) + a(y) + a(z)\big)
=
\sum_{v \in V}W'(v)a(v).
\]
Then, the left-hand side of \cref{equ::equinLemObjective} becomes $\sum_{v \in V} \left(W'(v) a(v) + (n+1) \beta(v)\right)$.

For $v \in C$, we have $a(v)=n+1$ and hence
\begin{align*}
W'(v)a(v)+(n+1)\beta(v)
= (n+1) \big(W'(v) + (1-W'(v))_+\big)
&=(n+1)\big(1+(W'(v)-1)_+\big)\\
&=a(v) + (n+1)(W'(v)-1)_+.
\end{align*}
For $v \in V \sm C$, note that $0 \le a(v) \le n+1$ by definition. Considering whether $W'(v) \ge 1$, we get
\[
W'(v)a(v)+(n+1)\beta(v)
= a(v) W'(v) + (n+1)(1-W'(v))_+
 \le a(v)+(n+1)|W'(v)-1|.
\]
Now, by \cref{equ::translatedWeightError} and the definition of $a(v)$, we have that the left-hand side of \cref{equ::equinLemObjective} is at most
\begin{align*}
& \sum_{v \in C} \left( a(v) + (n+1)(W'(v)-1)_+ \right) +
\sum_{v \in V \sm C} \left( a(v)+(n+1)|W'(v)-1| \right) \\
= & \sum_{v \in V} a(v) + (n+1) \Big ( \sum_{v \in C}(W'(v)-1)_+ + \sum_{v \in V \sm C}|W'(v)-1| \Big) \\
 \le & |S_{d,n}|+ (n+1) \cdot O_d(n^{d-2})
 \le M(Q)+O_d(n^{d-1}). \qedhere
\end{align*}
\end{proof}

\section{Proof of the fixed-profile proposition}
\label{sec::fixedProfileProof}

In this section, we prove \cref{prop::fixedProfile}. We first prove a local counting estimate for a fixed $(x,y,z) \in \T$, which includes two cases depending on whether $x=y$. We then combine all triples in $\T$ using a strong fractional entropy inequality of Madiman and Tetali~\cite{madiman2010information}.

For a finitely supported random variable $U$ with distribution $(p_u)_u$, its binary entropy is
\[
H(U) \ce -\sum_u p_u\log_2p_u.
\]
For finitely valued random variables $U_1,\ldots,U_k$, let $H(U_1,\ldots,U_k)$ denote the entropy of the joint random variable $(U_1,\ldots,U_k)$. For two random variables $U, U'$, let $H(U \mid U') \ce H(U, U')-H(U')$ be the conditional entropy.

\subsection{Local entropy estimates} \label{subsec::locEntropyEstimate}

\begin{proposition}[Local entropy estimates]\label{prop::localEntropy}
There are absolute constants $\alpha,\kappa>0$ such that the following hold for every integer $L \ge 0$.
\begin{enumerate}[(1)]
\item Fix $a,b \in [0,L+1]$ and $c \in [-1,L]$. For every random triple $(X,Y,Z)$ supported on triples of subsets of $[0,L]$ satisfying
\[
 (X+Y) \cap Z=\es, \quad
 \min(X \cup \{L+1\})=a, \quad
 \min(Y \cup \{L+1\})=b, \quad
 \max(Z \cup \{-1\})=c,
\]
we have
\[
H(X,Y,Z) \le 2L-\alpha|c-a-b|+\kappa.
\]

\item Fix $a \in [0,L+1]$ and $c \in [-1,L]$. For every random pair $(X,Z)$ supported on pairs of subsets of $[0,L]$ satisfying
\[
(X+X) \cap Z=\es,\qquad \min(X \cup \{L+1\})=a,\qquad \max(Z \cup \{-1\})=c,
\]
we have
\[
H(X,Z)+H(X\mid Z) \le 2L-\alpha|c-2a|+\kappa.
\]
\end{enumerate}
\end{proposition}

For a triple $(x,y,z) \in \T$ with $x \ne y$, the restrictions of a sum-free set to $F_x,F_y,F_z$ satisfy the hypothesis of \cref{prop::localEntropy}~(1). When $x=y=v$, our application of the strong fractional entropy inequality needs to use the blocks $\{v,2v\}$ and $\{v\}$. Since the target $2v$ precedes the source $v$ in the order used below, the second occurrence of the source needs to contribute only a conditional entropy. This is why part~(2) estimates $H(X,Z)+H(X\mid Z)$ rather than $H(X,Z)+H(X)$.

The repeated-source term can also be handled by ordinary fractional Shearer together with the Green--Morris theorem~\cite[Theorem~1.1]{green2016counting}. Two further alternatives are a modulo-$3$ refinement of fibers and an anti-collision refinement of the Keevash--Lim couplings. We use the strong fractional entropy inequality below because, together with Zhao's bipartite swapping lemma, it reduces part~(2) to part~(1) while keeping the definition of profiles. We discuss these alternatives in \cref{sec::concluding}.

We first prove part~(1). We use the following lemma due to Ghosal; see Definition~3.12 and Lemma~3.14 in~\cite{ghosal2025number}. For completeness, a self-contained proof, following Ghosal's main idea and simplifying the argument slightly, is given in \cref{app::GhosalLem}. For every integer $t \ge 0$, let $\S_t$ be the family of ordered pairs
\[
\S_t \ce
\bigl\{(S_1,S_2):S_1,S_2 \subseteq [0,t],\ 0\in S_1 \cap S_2,\ t+1\notin S_1+S_2\bigr\},
\]
and define
\[
g(t) \ce
\sum_{(S_1,S_2)\in\S_t}
2^{-| (S_1+S_2) \cap [0,t] |}.
\]

\begin{lemma}[Ghosal]\label{lem::Ghosal}
There are absolute constants $\eta>0$ and $K>0$ such that for every integer $t \ge 0$, we have
\[
g(t) \le K(2-\eta)^t.
\]
\end{lemma}

\begin{proof}[Proof of \cref{prop::localEntropy} (1)]
Fix $\alpha > 0$ to be sufficiently small and $\kappa > 0$ to be sufficiently large. Let $\Delta \ce c-a-b$. It suffices to prove that the total number of possible triples $(X,Y,Z)$ is at most
\[
2^{2L-\alpha|\Delta|+\kappa}.
\]

Suppose first that $\Delta \le 0$. A source set with prescribed minimum $a$ has at most $2^{L-a+1}$ choices, including the case $a=L+1$ when the source is empty. Similarly, the second source has at most $2^{L-b+1}$ choices, and a target set with prescribed maximum $c$ has at most $2^{c+1}$ choices, including the case $c=-1$ when the target is empty. Therefore, the number of possible triples is at most
\[
2^{L-a+1}2^{L-b+1}2^{c+1}
=2^{2L+\Delta+3}
=2^{2L-|\Delta|+3}.
\]
Hence, this case satisfies the required bound whenever $\alpha \le 1$ and $\kappa \ge 3$.

Now, assume that $\Delta>0$. In particular, $a,b \le L$ and $c \ge 0$, so all three sets are nonempty. For every possible triple $(X,Y,Z)$, let
\[
X' \ce \{i \in [0,\Delta-1] : a+i \in X\}, \qquad
Y' \ce \{j \in [0,\Delta-1] : b+j \in Y\}.
\]
Then, $0 \in X' \cap Y'$ and $\Delta \notin X'+Y'$, since $c=a+b+\Delta$ belongs to $Z$. Note that every triple $(X,Y,Z)$ can be determined by $(X',Y',Z)$ and $(X \sm (X'+a), Y \sm (Y'+b))$.

For $(X \sm (X'+a), Y \sm (Y'+b))$, since $a$ and $b$ are the respective minima, all smaller elements are absent. Hence, these two sets are in $[a+\Delta, L]$ and $[b+\Delta, L]$, respectively, so the number of choices is at most
\begin{equation} \label{equ::numXsmXpYsmYp}
2^{L-a-\Delta+1} \cdot 2^{L-b-\Delta+1}
=2^{2L-a-b-2\Delta+2}.
\end{equation}

For $(X',Y',Z)$, note that for every $s \in (X'+Y') \cap [0,\Delta-1]$, the point $a+b+s$ belongs to $X+Y$ and satisfies $a+b+s<c$. Since $(X+Y) \cap Z=\es$, none of these distinct points may belong to $Z$. Hence, the elements in $Z \sm \{c\}$ are chosen from $[0,c-1]$ after excluding the $|(X'+Y') \cap [0,\Delta-1]|$ forbidden points. Therefore, given $X',Y'$, the number of choices of $Z$ is at most
\begin{equation}
2^{c-|(X'+Y') \cap [0,\Delta-1]|}.  \label{equ::numXpYpZ}
\end{equation}

By \cref{equ::numXsmXpYsmYp,equ::numXpYpZ} and the definition of $g(\Delta-1)$, the total number of possible $(X,Y,Z)$ is at most
\[
2^{2L-a-b-2\Delta+2} \cdot 2^c \cdot
\sum_{(X',Y')\in\S_{\Delta-1}}
2^{-|(X'+Y') \cap [0,\Delta-1]|}
=2^{2L-\Delta+2} \cdot g(\Delta-1).
\]
By \cref{lem::Ghosal}, this is at most $2^{2L-\alpha\Delta+O(1)}$, where one may take any $\alpha<1-\log_2(2-\eta)$.
\end{proof}

Now, we prove part~(2). We use the following corollary of Zhao's bipartite swapping lemma; see Lemma~2.1 in~\cite{zhao2010independent}. Indeed, Lemma~2.1 in~\cite{zhao2010independent} actually gives a size-preserving bijection between ordered pairs of independent sets of $G$ and a subfamily of these cross-independent pairs $(A, B)$.

\begin{lemma} \label{lem::bipartiteSwapping}
Let $G$ be a finite graph, and let $i(G)$ be the number of independent sets of $G$. Then, the number of ordered pairs $(A,B)$ of subsets of $V(G)$ such that no edge of $G$ has one endpoint in $A$ and the other in $B$ is at least $i(G)^2$.
\end{lemma}

\begin{proof}[Proof of \cref{prop::localEntropy}~(2)]
Fix $a \in [0,L+1]$ and $c \in [-1,L]$. Let $\Z_c \ce \{T \subseteq [0,L]: \max(T \cup \{-1\})=c\}$. For every $T \in \Z_c$, let $p_T \ce \Pr(Z=T)$ and
\[
\U_T \ce
\{S \subseteq [0,L]:\min(S \cup \{L+1\})=a,\ (S+S) \cap T=\es\}.
\]
Note that the support of $X$ conditional on $Z=T$ is contained in $\U_T$, so
\[
H(X\mid Z=T) \le \log_2|\U_T|.
\]
Hence, we have
\begin{align*}
H(X,Z)&+H(X\mid Z)
 =H(Z)+2H(X\mid Z)
= \sum_{T\in \Z_c} p_T \cdot \log_2 \frac{1}{p_T}
 + 2 \sum_{T\in \Z_c} p_T \cdot H(X \mid Z = T) \\
&
 \le \sum_{T\in \Z_c} p_T \log_2 \frac{1}{p_T}
 + 2 \sum_{T\in \Z_c} p_T \log_2 |\U_T|
= \sum_{T\in \Z_c} p_T \log_2 \frac{|\U_T|^2}{p_T}
 \le \log_2 \Big( \sum_{T\in \Z_c} |\U_T|^2 \Big)
.
\end{align*}
Now, we give an upper bound on $|\U_T|^2$ using \cref{lem::bipartiteSwapping}.
\begin{claim}
For every $T \in \Z_c$, define \[ \U'_T \ce \big\{ (U_1,U_2) : U_1, U_2 \subseteq [0,L],\, (U_1 + U_2) \cap T = \es, \,\min(U_1 \cup \{L+1\}) = \min(U_2 \cup \{L+1\}) = a \big\}. \] Then, we have
    \[
        |\U_T|^2 \le |\U'_T|.
    \]
\end{claim}
\begin{proof}
Fix an arbitrary $T \in \Z_c$. For the case $a=L+1$, we have $\U_T=\{\es\}$, and the claim follows since $(\es,\es) \in \U'_T$. Hence, assume that $a \le L$. If $2a \in T$, then $\U_T = \es$ and the claim is trivially true. Hence, assume that $2a\notin T$. Define the graph $G_T$ on
\[
V(G_T) \ce \{u\in[a+1,L]:a+u\notin T,\ 2u\notin T\},
\]
where two distinct vertices $u$ and $v$ are adjacent if $u+v\in T$. Note that the map $I\mapsto\{a\} \cup I$ is a bijection between the independent sets of $G_T$ and the sets in $\U_T$. Indeed, if $U\in\U_T$, then $a\in U$, every $u\in U\sm\{a\}$ belongs to $V(G_T)$, and $U\sm\{a\}$ is independent in $G_T$. Conversely, if $I$ is independent in $G_T$, then none of the possible sums $2a$, $a+u$, $2u$, or $u+v$ with distinct $u,v\in I$ belongs to $T$, so $\{a\} \cup I\in\U_T$.

Then, by \cref{lem::bipartiteSwapping}, $|\U_T|^2$ is at most the number of ordered pairs $(A,B)$ of subsets of $V(G_T)$ with no edge between them. For every such pair $(A,B)$, let $U_1 \ce \{a\} \cup A$ and $U_2 \ce \{a\} \cup B$, and we claim
\begin{equation} \label{equ::U1U2}
\min(U_1 \cup \{L+1\})=
\min(U_2 \cup \{L+1\})=a
\quad\text{and}\quad
(U_1+U_2) \cap T=\es.
\end{equation}
Indeed, we have $2a\notin T$. The definition of $V(G_T)$ excludes sums of the form $a+u$ and, when $u\in A \cap B$, sums of the form $2u$; the absence of an edge between $A$ and $B$ excludes sums $u+v \in T$ with distinct $u\in A$ and $v\in B$. Therefore, $|\U_T|^2$ is at most the number of ordered pairs $(U_1,U_2)$ satisfying \cref{equ::U1U2}, and all these pairs belong to $\U'_T$. This proves the claim.
\end{proof}

By the claim,
\[
H(X,Z)+H(X\mid Z)
 \le
\log_2\sum_{T\in\Z_c}|\U'_T|.
\]
The last sum is exactly the number of ordered triples $(U_1,U_2,T)$ satisfying
\[
(U_1+U_2) \cap T=\es,\qquad
\min(U_1 \cup \{L+1\})=
\min(U_2 \cup \{L+1\})=a,\qquad
\max(T \cup \{-1\})=c.
\]
Let $(X',Y',Z')$ be uniformly distributed over these triples. Applying \cref{prop::localEntropy}~(1) with both source minima equal to $a$, we obtain
\[
H(X,Z)+H(X\mid Z)
 \le
H(X',Y',Z')
 \le
2L-\alpha|c-2a|+\kappa,
\]
as required.
\end{proof}

\subsection{Completing the fixed-profile count} \label{subsec::completingFixProfileCount}

Now, we prove \cref{prop::fixedProfile}. We use the following ordered conditional fractional form of Shearer's inequality, due to Madiman and Tetali, see Theorem~I' in~\cite{madiman2010information}.

\begin{lemma}\label{lem::shearer}
Let $I$ be a finite set with a total order, and let $(X_v)_{v \in I}$ be finitely valued random variables. For $U \subseteq I$, let $I_{<U}$ be the set of elements of $I$ that precede every element of $U$. If $\theta_U \ge 0$ for $U \subseteq I$ and $\sum_{U: v \in U }\theta_U \ge 1$ for every $v \in I$, then
\[
H((X_v)_{v \in I})
 \le
\sum_{U \subseteq I}\theta_U
H\bigl((X_v)_{v \in U}\mid (X_v)_{v \in I_{<U}}\bigr).
\]
\end{lemma}

\begin{proof}[Proof of \cref{prop::fixedProfile}]
If $N(\P)=0$, then the conclusion is trivial. Hence, we may assume that $N(\P)>0$. Let $S$ be the random variable that is uniformly distributed over the $N(\P)$ sum-free sets with profile $\P$, and let $ X_v \ce S_v$ for every $v\in V$. Since the map $S\mapsto( X_v)_{v\in V}$ is bijective, we have
\[
\log_2N(\P) = H(S) = H(( X_v)_{v\in V}).
\]

Let $\prec$ be the order on $V$ such that $h(u)>h(v)$ implies $u\prec v$, breaking ties arbitrarily. For $U \subseteq V$, write $ X_{<U}$ for the collection of variables $ X_{v'}$ such that $v'$ precedes every element of $U$, and write $ X_{<v}$ for $ X_{<\{v\}}$.

Choose one of the weighted defect sums whose value is $\delta(\P)$. If it is $\delta_i^{r,s}(\P)$ for some $i\in\{1,2\}$ and $(r,s)\in\R$, let $\o'$ be obtained from $\o$ by replacing $\o_i$ with $\o_i^{r,s}$ and leaving the other two components unchanged. If it is $\delta_3(\P)$, let $\o' \ce \o$; note that this case is also included in \cref{lem::objective} by letting $i=1$ and $(r,s)=(0,0)$. In either case, let $W' \ce W_{\o'}$ and $\beta(v) \ce (1-W'(v))_+$.

We apply \cref{lem::shearer} with $I = V$, ordering $ \prec$ and the $U, \theta_U$ defined as follows\footnote{Formally, we extend $\o'$ to out of its support by $0$ and only include those sets $U$ with positive $\theta_U$.}.
\begin{itemize}
\item For every $U = \{x,y,z\}$ with distinct $x,y,z$, let $\theta_U = \o'(x,y,z) + \o'(x,z,y) + \o'(y,x,z) + \o'(y,z,x) + \o'(z,x,y) + \o'(z,y,x)$.
\item For every $U = \{x,2x\}$, let $\theta_U = \o'(x,x,2x)$.
\item For every $U = \{x\}$, let $\theta_U = \o'(x,x,2x) + \beta(x)$.
\end{itemize}
Then, for every $v \in V$, we have $\sum_{U:v\in U} \theta_U= W'(v) + \beta(v) \ge 1$. Hence, we can apply \cref{lem::shearer} and get
\begin{alignat*}{2}
H(( X_v)_{v\in V})
& \le {}&\phantom{{}+{}}\qquad&
\sum_{\mathclap{\substack{(x,y,z)\in\T\\
\o'(x,y,z)>0,\ x\ne y}}}
\o'(x,y,z)
H( X_x, X_y, X_z
\mid X_{<\{x,y,z\}})\\
&&{}+{}\qquad&
\sum_{\mathclap{\substack{v\in V\\
\o'(v,v,2v)>0}}}
\o'(v,v,2v)
\Bigl(
H( X_v, X_{2v}
\mid X_{<\{v,2v\}})
+H( X_v\mid X_{<v})
\Bigr)\\
&&{}+{}\qquad&
\sum_{\mathclap{v\in V}}
\beta(v)
H( X_v\mid X_{<v}).
\end{alignat*}

Let $\alpha,\kappa>0$ be the absolute constants given by \cref{prop::localEntropy}.

For every triple $(x,y,z)\in\T$ with $x\ne y$ and $\o'(x,y,z)>0$, the support properties of $\o'$ imply that $x,y,z$ are distinct. Since every realization of $S$ is sum-free and has profile $\P$, the random triple $( X_x, X_y, X_z)$ satisfies the hypotheses of \cref{prop::localEntropy}~(1), with source minima $\P(x),\P(y)$ and target maximum $\P(z)$. Hence, by \cref{prop::localEntropy}~(1),
\[
H( X_x, X_y, X_z \mid X_{<\{x,y,z\}})
 \le
H( X_x, X_y, X_z)
 \le
2n-\alpha|\P(z)-\P(x)-\P(y)|+\kappa.
\]

For every $(v,2v)$ with $\o'(v,v,2v)>0$, the support properties give $v\in B \cup C^-$, so $h(v)>0$ and hence $2v\prec v$. The random pair $( X_v, X_{2v})$ satisfies the hypotheses of \cref{prop::localEntropy}~(2), with source minimum $\P(v)$ and target maximum $\P(2v)$. We have
\[
H( X_v, X_{2v}
\mid X_{<\{v,2v\}})
+H( X_v\mid X_{<v}) \le
H( X_v, X_{2v})
+H( X_v\mid X_{2v})
 \le
2n-\alpha|\P(2v)-2\P(v)|+\kappa.
\]

For the singletons, we trivially have $H( X_v\mid X_{<v}) \le H( X_v) \le n+1$.

By our choice of $\o'$, we have
\[
\sum_{(x,y,z)\in\T}
\o'(x,y,z)|\P(z)-\P(x)-\P(y)|
 \ge \delta(\P).
\]
Indeed, if $\delta(\P)=\delta_i^{r,s}(\P)$, then the contribution of $\o_i^{r,s}$ to the left-hand side is exactly $\delta_i^{r,s}(\P)$ by definition; the case $\delta(\P)=\delta_3(\P)$ is immediate by definition. Also, we have $\sum_{(x,y,z)\in\T}\o'(x,y,z)=O_d(n^{d-1})$, which follows from the total-mass bounds in \cref{lem::w1Weights,lem::w2Weights,lem::w3Weights}, as a translated component has no larger total mass than the corresponding original component.

Thus, by \cref{lem::objective}, we have
\begin{align*}
\log_2N(\P)
& \le
\sum_{(x,y,z)\in\T}\o'(x,y,z)
\bigl(2n-\alpha|\P(z)-\P(x)-\P(y)|+\kappa\bigr)
+(n+1)\sum_{v\in V}\beta(v)\\
& \le
2(n+1)\sum_{(x,y,z)\in\T}\o'(x,y,z)
+(n+1)\sum_{v\in V}\beta(v)
-\alpha\delta(\P)
+O_d(n^{d-1})\\
& \le
M(Q)-\alpha\delta(\P)+O_d(n^{d-1}),
\end{align*}
Taking $\alpha_0=\alpha$ completes the proof.
\end{proof}

\section{Proof of the profile-count proposition}
\label{sec::profileCountProof}
In this section, we prove \cref{prop::profileCount}. Recall that in \cref{subsec::profileanddefect}, we define that for every profile, the lower part is the restriction to $A \cup B \cup C^-$, and the upper part is the restriction to $C^+ \cup D \cup E$. We prove \cref{prop::profileCount} by the following two lemmas dealing with these two parts separately.

\begin{lemma}[Lower profiles]\label{lem::lowerProfileCount}
For every $\eta>0$ and $R \ge 0$, the number of functions on $A \cup B \cup C^-$ that extend to at least one profile $\P$ satisfying $\delta(\P) \le R$ is at most
\[
2^{\eta R+O_{d,\eta}(n^{d-1})}.
\]
\end{lemma}

\begin{lemma}[Upper profiles]\label{lem::upperProfileCount}
Fix the restriction of a profile to $A \cup B \cup C^-$. For every $\eta>0$ and $R \ge 0$, the number of functions on $C^+ \cup D \cup E$ completing it to a profile $\P$ with $\delta(\P) \le R$ is at most
\[
2^{\eta R+O_{d,\eta}(n^{d-1})}.
\]
The estimate is uniform over the fixed lower restriction.
\end{lemma}

\begin{proof}[Proof of \cref{prop::profileCount} assuming \cref{lem::lowerProfileCount,lem::upperProfileCount}]
Fix $R \ge 0$ and apply both lemmas with $\eta=\eps/2$. There are at most $2^{\eps R/2+O_{d,\eps}(n^{d-1})}$ possible lower parts. Uniformly for each such lower part, there are at most $2^{\eps R/2+O_{d,\eps}(n^{d-1})}$ possible upper parts. Multiplying these two bounds gives
\[
\bigl|\{\P : \delta(\P) \le R\}\bigr|
 \le 2^{\eps R+O_{d,\eps}(n^{d-1})}. \qedhere
\]
\end{proof}

For \cref{lem::lowerProfileCount}, a signed combination of four translated defects cancels the target term and one source term and isolates a second difference on the remaining source. Along each line parallel to $e_1$, up to two initial values together with all second differences determine the profile, so these bounds allow us to count the lower part. For \cref{lem::upperProfileCount}, once the lower part is fixed, the unshifted defects give a separate weighted bound at each point of $C^+$ and $D$, and at each point of $E$ with positive target coordinate sum; the remaining $O_d(n^{d-2})$-point boundary layer of $E$ is counted freely.

We remark that this cancellation-and-reconstruction argument is inspired by Ghosal's two-dimensional argument in Section~3.5, especially Definition~3.24 and Lemmas~3.26 and~3.27, in~\cite{ghosal2025number}. We make this method work in every fixed dimension $d \ge 3$ using the nonuniform Keevash--Lim weights and bounded translations, while controlling the boundary losses created by these translations.
\subsection{Proof of the lower-profile lemma}

Given a function $f$, for every $v\in V$ for which
$f(v)$, $f(v+e_1)$, and $f(v+2e_1)$ are defined, let
\[
\Delta_{e_1}^2f(v)
\ce
f(v+2e_1)-2f(v+e_1)+f(v).
\] 
The sets $\G_1$ and $\G_2$ defined in \cref{subsec::transWeights} ensure that all shifted defects used below are defined on the same original triple. Recall that we define in \cref{def::localDefect} the local defects $\delta_1^{r,s}(x,y,z)$, $\delta_2^{r,s}(x,y,z)$, and $\delta_3(x,y,z)$.

\begin{lemma}[Cancellation identities]
\label{lem::translationCancellations}
Fix a profile $\P$. For every $(x,y,z) \in \G_1$, we have
\[
\Delta_{e_1}^2\P(x)=-\delta_1^{2e_1,0}(x,y,z)+\delta_1^{e_1,e_1}(x,y,z)+\delta_1^{e_1,0}(x,y,z)-\delta_1^{0,e_1}(x,y,z).
\]
For every $(x,y,z) \in \G_2$, we have
\begin{align*}
\Delta_{e_1}^2\P(x) &=-\delta_2^{2e_1,0}(x,y,z)+\delta_2^{e_1,e_1}(x,y,z)+\delta_2^{e_1,0}(x,y,z)-\delta_2^{0,e_1}(x,y,z), \\
\Delta_{e_1}^2\P(y) &=-\delta_2^{0,2e_1}(x,y,z)+\delta_2^{e_1,e_1}(x,y,z)+\delta_2^{0,e_1}(x,y,z)-\delta_2^{e_1,0}(x,y,z).
\end{align*}
\end{lemma}

\begin{proof}
For $(x,y,z) \in \G_1$, all four defects in the first identity are defined on the same original triple. Expanding them by definition gives
\begin{align*}
 &-\delta_1^{2e_1,0}(x,y,z)+\delta_1^{e_1,e_1}(x,y,z)
   +\delta_1^{e_1,0}(x,y,z)-\delta_1^{0,e_1}(x,y,z)\\
 &=-\bigl[\P(z+2e_1)-\P(x+2e_1)-\P(y)\bigr]
   +\bigl[\P(z+2e_1)-\P(x+e_1)-\P(y+e_1)\bigr]\\
 &\quad+\bigl[\P(z+e_1)-\P(x+e_1)-\P(y)\bigr]
   -\bigl[\P(z+e_1)-\P(x)-\P(y+e_1)\bigr]\\
 &=\P(x+2e_1)-2\P(x+e_1)+\P(x)
 =\Delta_{e_1}^2\P(x).
\end{align*}
For $(x,y,z) \in \G_2$,
\begin{align*}
 &-\delta_2^{2e_1,0}(x,y,z)+\delta_2^{e_1,e_1}(x,y,z)
   +\delta_2^{e_1,0}(x,y,z)-\delta_2^{0,e_1}(x,y,z)\\
 &=-\bigl[\P(z+2e_1)-\P(x+2e_1)-\P(y)\bigr]
   +\bigl[\P(z+2e_1)-\P(x+e_1)-\P(y+e_1)\bigr]\\
 &\quad+\bigl[\P(z+e_1)-\P(x+e_1)-\P(y)\bigr]
   -\bigl[\P(z+e_1)-\P(x)-\P(y+e_1)\bigr]\\
 &=\P(x+2e_1)-2\P(x+e_1)+\P(x)
 =\Delta_{e_1}^2\P(x),
\end{align*}
and
\begin{align*}
 &-\delta_2^{0,2e_1}(x,y,z)+\delta_2^{e_1,e_1}(x,y,z)
   +\delta_2^{0,e_1}(x,y,z)-\delta_2^{e_1,0}(x,y,z)\\
 &=-\bigl[\P(z+2e_1)-\P(x)-\P(y+2e_1)\bigr]
   +\bigl[\P(z+2e_1)-\P(x+e_1)-\P(y+e_1)\bigr]\\
 &\quad+\bigl[\P(z+e_1)-\P(x)-\P(y+e_1)\bigr]
   -\bigl[\P(z+e_1)-\P(x+e_1)-\P(y)\bigr]\\
 &=\P(y+2e_1)-2\P(y+e_1)+\P(y)
 =\Delta_{e_1}^2\P(y). \qedhere
\end{align*}
\end{proof}

Let
\[
A_1 \ce \{v \in A : 1 \le h(v) \le m-n-2\}.
\]
By \cref{lem::w2Weights}, the total $\o_2$-weight at the first source equals $1$ precisely on $A_1$. The complement $A \sm A_1$ consists of the two layers $h(v)=0$ and $h(v)=m-n-1$, and hence has size $O_d(n^{d-2})$.

\begin{lemma}[Second differences of the lower profile]
\label{lem::profileSecondDifferences}
For every profile $\P$,
\begin{equation} \label{equ::lowerSecondDifferences}
\begin{aligned}
 \sum_{\substack{v:v \in B\\v+e_1,v+2e_1 \in B}}
 |\Delta_{e_1}^2\P(v)|
 &=O_d(\delta(\P)+n^{d-1}),\\
 \sum_{\substack{v:v \in A_1\\v+e_1,v+2e_1 \in A_1}}
 |\Delta_{e_1}^2\P(v)|
 &=O_d(\delta(\P)+n^{d-1}),\\
 \sum_{\substack{v:v \in C^-\\v+e_1,v+2e_1 \in C^-}}
 \lambda_-(v)|\Delta_{e_1}^2\P(v)|
 &=O_d(\delta(\P)+n^{d-1}).
\end{aligned}.
\end{equation}
\end{lemma}

\begin{proof}
Since $\G_i \subseteq \G_i^{r,s}$ for every $(r,s) \in \R$, multiplying the three identities in \cref{lem::translationCancellations} by the original weights and taking the sum gives
\begin{align}
\sum_{(x,y,z) \in \G_1}
\o_1(x,y,z)|\Delta_{e_1}^2\P(x)|
& \le 4\delta(\P),\label{equ::sumxyzG1o1Delta}\\
\sum_{(x,y,z) \in \G_2}
\o_2(x,y,z)|\Delta_{e_1}^2\P(x)|
& \le 4\delta(\P),\label{equ::sumxyzG2o2DeltaPx}\\
\sum_{(x,y,z) \in \G_2}
\o_2(x,y,z)|\Delta_{e_1}^2\P(y)|
& \le 4\delta(\P). \label{equ::sumxyzG2o2DeltaPy}
\end{align}
Indeed, we have
\[
\sum_{(x,y,z) \in \G_1} \o_1(x,y,z) \big|\delta_1^{2e_1,0}(x,y,z)\big|
 \le
\delta_1^{2e_1,0}(\P) \le \delta(\P),
\]
by definition and similar arguments hold for all other terms.

For the first estimate in \cref{equ::lowerSecondDifferences}, by \cref{equ::w1CoordinateSums}, we have  $\sum_{y,z \in V}\o_1(v,y,z) \ge 1/3$ and $\sum_{x,z \in V}\o_1(x,v,z) \ge 1/3$. Also, note that trivially $|\Delta_{e_1}^2\P(v)| \le 2n+2$ for every $v \in B$. By \cref{equ::sumxyzG1o1Delta,equ::commonValidMass}, We have
\begin{align*}
\frac{1}{3}
\sum_{\substack{v: v \in B\\v+e_1,v+2e_1 \in B}}
|\Delta_{e_1}^2\P(v)|
& \le
\sum_{\substack{(x,y,z) \in \supp(\o_1)\\
                  x \in B,\ x+e_1,x+2e_1 \in B}}
\o_1(x,y,z)|\Delta_{e_1}^2\P(x)|\\
& \le
\sum_{(x,y,z) \in \G_1}
\o_1(x,y,z)|\Delta_{e_1}^2\P(x)|+(2n+2)
\sum_{(x,y,z) \in \B_1}
\o_1(x,y,z)\\
&=O_d(\delta(\P)+n^{d-1}).
\end{align*}

For the second estimate in \cref{equ::lowerSecondDifferences}, by \cref{lem::w2Weights}, we have $\sum_{y,z \in V}\o_2(v,y,z) = 1$ for every $v \in A_1$. The same argument gives
\begin{align*}
\sum_{\substack{v:v \in A_1\\v+e_1,v+2e_1 \in A_1}}
|\Delta_{e_1}^2\P(v)|
& \le
\sum_{(x,y,z) \in \G_2}
\o_2(x,y,z)|\Delta_{e_1}^2\P(x)|+(2n+2)
\sum_{(x,y,z) \in \B_2}
\o_2(x,y,z)\\
&=O_d(\delta(\P)+n^{d-1}).
\end{align*}

For the third estimate in \cref{equ::lowerSecondDifferences}, by definition, we have $\lambda_-(v) = \sum_{x,z \in V}\o_2(x,v,z)$ for every $v \in C^-$. Hence, by \cref{lem::w2Weights},
\begin{align*}
\sum_{\substack{v:v \in C^-\\v+e_1,v+2e_1 \in C^-}}
\lambda_-(v)|\Delta_{e_1}^2\P(v)|
& \le
\sum_{(x,y,z) \in \G_2}
\o_2(x,y,z)|\Delta_{e_1}^2\P(y)|
+(2n+2)\sum_{(x,y,z) \in \B_2}
\o_2(x,y,z)\\
&=O_d(\delta(\P)+n^{d-1}). \qedhere
\end{align*}
\end{proof}

\begin{proof}[Proof of \cref{lem::lowerProfileCount}]
Fix $\eta > 0$ and $R \ge 0$. For every $u=(u_2,\ldots,u_{d-1}) \in [0,n]^{d-2}$, let
\[
L_u \ce \{(t,u_2,\ldots,u_{d-1}) : t \in [0,n]\}.
\]
These $(n+1)^{d-2}$ lines parallel to $e_1$ partition $V$. Since $A_1$, $B$, and $C^-$ are each defined by an interval of values of $h$, their intersections with every $L_u$ are intervals, possibly empty. On each nonempty interval, any function $f$ is determined by its first two values and all its second differences $\Delta_{e_1}^2f(v)$, as
\[
f(v+2e_1)=2f(v+e_1)-f(v)+\Delta_{e_1}^2f(v).
\]
The case $d=3$ is illustrated in \cref{fig::profileLineEncoding}.

\begin{figure}[ht]
\centering
\begin{tikzpicture}[
  x=1.15cm,y=1.15cm,
  line cap=round,line join=round,
  every node/.style={font=\small}
]
  \fill[black!18] (0,0)--(1.2,0)--(0,1.2)--cycle;
  \fill[black!5]
    (1.2,0)--(6,0)--(6,0.8)--(0.8,6)--(0,6)--(0,1.2)--cycle;
  \fill[black!12]
    (6,0.8)--(6,1.8)--(1.8,6)--(0.8,6)--cycle;

  \draw[semithick] (0,0) rectangle (6,6);
  \draw (0,1.2)--(1.2,0);
  \draw (6,0.8)--(0.8,6);
  \draw (6,1.8)--(1.8,6);

  \draw[-{Latex[length=2mm]},semithick]
    (0.08,0.35)--(0.78,0.35);
  \draw[-{Latex[length=2mm]},semithick]
    (0.08,2.60)--(4.12,2.60)
    node[midway,above] {$e_1$};
  \draw[-{Latex[length=2mm]},semithick]
    (2.17,4.70)--(3.03,4.70);

  \fill (0.18,0.35) circle (1.5pt);
  \fill (0.38,0.35) circle (1.5pt);
  \fill (0.18,2.60) circle (1.5pt);
  \fill (0.38,2.60) circle (1.5pt);
  \fill (2.27,4.70) circle (1.5pt);
  \fill (2.47,4.70) circle (1.5pt);

  \node at (0.28,0.58) {$A_1$};
  \node at (2.85,2.10) {$B$};
  \node at (3.25,4.05) {$C^-$};
  \node[anchor=north east] at (5.92,5.92) {$V$};
\end{tikzpicture}
\caption{Schematic picture for $d=3$. The profile on each line can be determined from the first two values and the second difference.}
\label{fig::profileLineEncoding}
\end{figure}
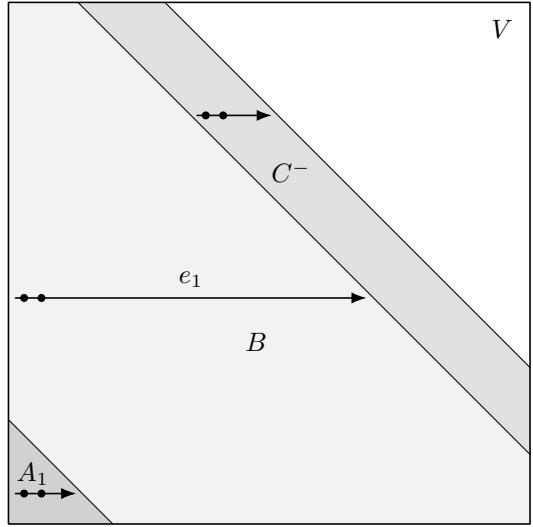

There are $(n+1)^{d-2}$ choices for $u \in [0,n]^{d-2}$. For each $u$, there are at most three intervals, corresponding to $A_1,B,C^-$, respectively. Each interval requires at most two initial profile values, and every such value has at most $n+2$ choices. Also, the values on $A \sm A_1$ may be chosen freely, since $|A \sm A_1|=O_d(n^{d-2})$ and each point has at most $n+2$ choices. Putting all these together, they contribute to at most $2^{O_d(n^{d-2}\log n)}$ choices.

It remains to count the second-difference.  By \cref{equ::lowerSecondDifferences}, there is a constant $\Gamma_d>0$ such that each of the three sums in \cref{equ::lowerSecondDifferences} is at most $\Gamma_d(\delta(\P)+n^{d-1})$. Let $\theta \ce \eta / (3\Gamma_d)$. Let $B' \ce \{v\in B: v+e_1,v+2e_1 \in B\}$, $A' \ce \{v\in A_1: v+e_1,v+2e_1 \in A_1\}$, and $C' \ce \{v\in C^-: v+e_1,v+2e_1 \in C^-\}$. Let $\F_{B'}$, $\F_{A'}$, and $\F_{C'}$ be the families of arrays
\[
\bigl(\Delta_{e_1}^2\P(v)\bigr)_{v\in B'},\qquad
\bigl(\Delta_{e_1}^2\P(v)\bigr)_{v\in A'},\qquad
\bigl(\Delta_{e_1}^2\P(v)\bigr)_{v\in C'},
\]
respectively, that arise from profiles $\P$ on $V$ with $\delta(\P) \le R$.

First, consider $\F_{B'}$. By the first estimate in \cref{equ::lowerSecondDifferences}, for every profile $\P$ with $\delta(\P) \le R$, we have
\[
\sum_{v \in B'}|\Delta_{e_1}^2\P(v)| \le \Gamma_d(R+n^{d-1}),
\]
so for every $f \in \F_{B'}$, we have
\[
    2^{\theta \left( \Gamma_d(R+n^{d-1}) - \sum_{v \in B'} |f(v)|  \right)} \ge 1.
\]
Hence,
\[
    \left| \F_{B'} \right|
    =
    \sum_{f \in \F_{B'}} 1
 \le
    \sum_{f \in \F_{B'}} 2^{\theta \left( \Gamma_d(R+n^{d-1}) - \sum_{v \in B'} |f(v)|  \right)}
    =
    2^{\theta \left( \Gamma_d(R+n^{d-1})\right)}
    \cdot
    \sum_{f \in \F_{B'}}
     \prod_{v \in B'} 2^{-\theta|f(v)|}.
\]
For every $v \in B'$, the value $f(v)$ is an integer. We have
\[
\sum_{f \in \F_{B'}}
     \prod_{v \in B'} 2^{-\theta|f(v)|}
 \le
\prod_{v \in B'} \sum_{i \in \Zs} 2^{-\theta |i|} = \prod_{v \in B'} O_{d,\eta}(1) = 2^{O_{d,\eta} (n^{d-1})}.
\]
Therefore, we have
\[
|\F_{B'}| \le 2^{\theta \left( \Gamma_d(R+n^{d-1})\right)} \cdot 2^{O_{d,\eta} (n^{d-1})}
=
2^{\theta\Gamma_d R  + O_{d,\eta} (n^{d-1})}
=
2^{\eta R/3  + O_{d,\eta} (n^{d-1})}.
\]

For $\F_{A'}$, similar computation gives
\[
|\F_{A'}| \le 2^{\eta R/3  + O_{d,\eta} (n^{d-1})}.
\]

For $\F_{C'}$, the third estimate in \cref{equ::lowerSecondDifferences} gives, for every $f\in\F_{C'}$,
\[
\sum_{v\in C'}\lambda_-(v)|f(v)|
 \le
\Gamma_d(R+n^{d-1}),
\]
so
\[
2^{\theta\left(
\Gamma_d(R+n^{d-1})-
\sum_{v\in C'}\lambda_-(v)|f(v)|
\right)}
 \ge 1.
\]
Hence,
\[
|\F_{C'}|
=
\sum_{f\in\F_{C'}}1
 \le
2^{\theta\Gamma_d(R+n^{d-1})}
\sum_{f\in\F_{C'}}
\prod_{v\in C'}2^{-\theta\lambda_-(v)|f(v)|}
 \le
2^{\theta\Gamma_d(R+n^{d-1})}
\prod_{v\in C'}
\sum_{i\in\Zs}2^{-\theta\lambda_-(v)|i|}.
\]
Note that for every $0<\lambda \le 1$, we have $1-2^{-\theta\lambda} \ge \lambda(1-2^{-\theta})$ by the concavity of $u\mapsto1-2^{-\theta u}$ on $[0,1]_{\Rs}$, so
\[
\sum_{i\in\Zs}2^{-\theta\lambda|i|}
=
\frac{1+2^{-\theta\lambda}}{1-2^{-\theta\lambda}}
 \le
\frac{2}{\lambda(1-2^{-\theta})}.
\]
Moreover, by \cref{equ::lambdaLogSum},
\[
\log_2\bigg(
\prod_{v\in C'}
\frac{2}{\lambda_-(v)(1-2^{-\theta})}
\bigg)
=
|C'|\log_2\frac{2}{1-2^{-\theta}}
+
\sum_{v\in C'}\log_2\frac{1}{\lambda_-(v)}=
O_{d,\eta}(n^{d-1}).
\]
Therefore,
\[
|\F_{C'}|
 \le
2^{\theta\Gamma_dR+O_{d,\eta}(n^{d-1})}
=
2^{\eta R/3+O_{d,\eta}(n^{d-1})}.
\]

Thus, the three second-difference arrays together have at most
\begin{align*}
|\F_{B'}|\,|\F_{A'}|\,|\F_{C'}|
& \le
2^{\eta R+O_{d,\eta}(n^{d-1})}
\end{align*}
possible values. This completes the proof of \cref{lem::lowerProfileCount}.
\end{proof}

\subsection{Proof of the upper-profile lemma}

\begin{proof}[Proof of \cref{lem::upperProfileCount}]
Fix the restriction of $\P$ to $A \cup B \cup C^-$ and consider completions with $\delta(\P) \le R$. Let $E^\ast \ce \{z \in E:h(z) \ge 2m+2\}$ and $U \ce C^+ \cup D \cup E^\ast$. Note that $E\sm E^\ast=\{z\in E:h(z)=2m+1\}$ has $O_d(n^{d-2})$ points, and hence its profile values may contribute at most $2^{O_d(n^{d-2}\log n)}$ choices. Hence, it suffices to consider the total number of upper-profile completions on $U$.

The main idea of the proof is as follows. For every $z\in U$, the unshifted defects compare $\P(z)$ with sums $\P(x)+\P(y)$ at source points $x,y\in A \cup B \cup C^-$, whose profile values are already fixed. We sum these pointwise bounds and count the possible values of the upper profile directly by a geometric series. This pointwise mechanism is illustrated in \cref{fig::upperProfileD} for $d=3$.

\begin{figure}[ht]
\centering
\begin{tikzpicture}[
  x=1.15cm,y=1.15cm,
  line cap=round,line join=round,
  every node/.style={font=\small},
  point/.style={circle,fill=black,inner sep=1.7pt}
]
  \fill[black!5]
    (1.2,0)--(6,0)--(6,0.8)--(0.8,6)--(0,6)--(0,1.2)--cycle;
  \fill[black!12]
    (1.8,6)--(6,6)--(6,1.8)--cycle;

  \draw[black!55] (0,1.2)--(1.2,0);
  \draw[black!55] (0.8,6)--(6,0.8);
  \draw[black!55] (1.8,6)--(6,1.8);
  \draw[semithick] (0,0) rectangle (6,6);

  \node at (2.85,2.10) {$B$};
  \node at (4.85,4.85) {$D$};
  \node[anchor=north east] at (5.92,5.92) {$V$};

  \coordinate (xpt) at (3.00,1.20);
  \coordinate (ypt) at (1.20,2.80);
  \coordinate (zpt) at (4.20,4.00);
  \draw[densely dashed,black!65] (xpt)--(zpt) (ypt)--(zpt);
  \node[point] at (xpt) {};
  \node[below right] at (xpt) {$x$};
  \node[point] at (ypt) {};
  \node[above left] at (ypt) {$y$};
  \node[point] at (zpt) {};
  \node[right] at (zpt) {$z=x+y$};
\end{tikzpicture}
\caption{A schematic picture for $d=3$ showing how the upper profile
on $D$ is counted after the lower profile on $B$ has been fixed. For
a fixed $z\in D$, every representation $z=x+y$ with $x,y\in B$ gives
the known $\P(x)+\P(y)$. Thus, the weighted sum
$\psi_z(\P(z))$ controls the possible value of $\P(z)$.}
\label{fig::upperProfileD}
\end{figure}
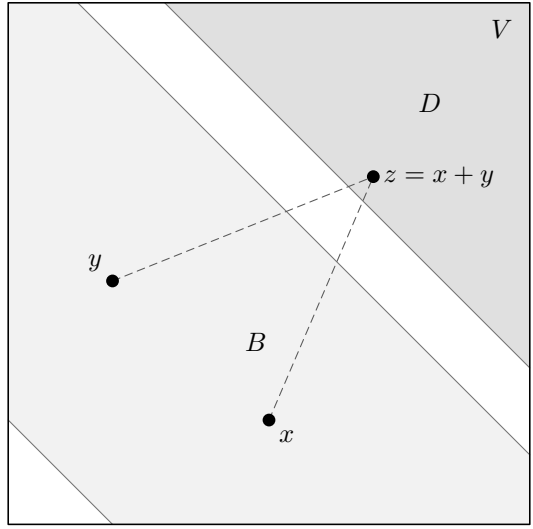

For every $z \in U$, let
\[
\lambda(z) \ce
\begin{cases}
\lambda_+(z), & z \in C^+,\\
1, & z \in D \cup E^\ast.
\end{cases}
\]
Note that $0 < \lambda(z) \le 1$ by the
definition of $\lambda$ and \cref{lem::w2Weights}. For every $t\in[-1,n]$, define
\[
\psi_z(t) \ce
\begin{cases}
\displaystyle
\sum_{\substack{x,y \in V\\x+y=z}}
\o_1(x,y,z)|t-\P(x)-\P(y)|,
&z \in D,\\[2mm]
\displaystyle
\sum_{\substack{x,y \in V\\x+y=z}}
\o_2(x,y,z)|t-\P(x)-\P(y)|,
&z \in C^+,\\[2mm]
\displaystyle
\sum_{\substack{x,y \in V\\x+y=z}}
\o_3(x,y,z)|t-\P(x)-\P(y)|,
&z \in E^\ast.
\end{cases}.
\]
By the support statements in \cref{lem::KLWeights}, all profile values on the right-hand side are over $A \cup B \cup C^-$ and hence fixed. 
Also, note that by
\cref{lem::w1Weights,lem::w2Weights,lem::w3Weights}, the total coefficient in the definition of $\psi_z$ is
$\lambda(z)$. By the triangle inequality,
\[
|t_1-\P(x)-\P(y)|+|t_2-\P(x)-\P(y)| \ge \bigl|(t_1-\P(x)-\P(y))-(t_2-\P(x)-\P(y))\bigr|=|t_1-t_2|.
\]
Multiplying by the corresponding weight and summing over all $x,y$ with $x+y=z$, we obtain
$\psi_z(t_1)+\psi_z(t_2) \ge \lambda(z)|t_1-t_2|$. For every $z\in U$, fix $t_z\in[-1,n]$ minimizing $\psi_z$. Since $\psi_z(t_z) \le \psi_z(t)$, the previous inequality gives
\begin{equation} \label{equ::psitge}
\psi_z(t)
 \ge
\frac{\lambda(z)}{2}|t-t_z|.
\end{equation}

By \cref{lem::w3Weights}, we have $\sum_{x,y \in V}\o_3(x,y,z)=0$ for every $z \in E \sm E^\ast$. Therefore, every completion $\P$ under consideration satisfies
\[
\sum_{z\in U}\psi_z(\P(z))
=
\delta_1^{0,0}(\P)+\delta_2^{0,0}(\P)+\delta_3(\P)
 \le
3\delta(\P)
 \le
3R,
\]
so by \cref{equ::psitge},
\begin{equation} \label{equ::sumzinUlambdapzmtz}
\sum_{z \in U}\lambda(z)|\P(z)-t_z|
 \le
2\sum_{z \in U}\psi_z(\P(z))
 \le 6R.
\end{equation}

Let $\theta \ce \eta/6$, and let $\F_U$ be the family of functions $f:U \to [-1,n]$ that arise as the restrictions of completions under consideration. For every $f\in\F_U$, by \cref{equ::sumzinUlambdapzmtz}, we have
\[
\sum_{z\in U}\lambda(z)|f(z)-t_z| \le 6R,
\]
so
\[
2^{\theta\left(
6R-\sum_{z\in U}\lambda(z)|f(z)-t_z|
\right)} \ge 1.
\]
Hence,
\[
|\F_U|
=
\sum_{f\in\F_U}1
 \le
2^{6\theta R}
\sum_{f\in\F_U}
\prod_{z\in U}2^{-\theta\lambda(z)|f(z)-t_z|}
 \le
2^{6\theta R}
\prod_{z\in U}
\sum_{i\in\Zs}2^{-\theta\lambda(z)|i|}.
\]
As in the proof of \cref{lem::lowerProfileCount}, for every $0<\lambda \le 1$, we have
\[
\sum_{i\in\Zs}2^{-\theta\lambda|i|}
=
\frac{1+2^{-\theta\lambda}}{1-2^{-\theta\lambda}}
 \le
\frac{2}{\lambda(1-2^{-\theta})}.
\]
Moreover, by the fact that $|U|=O_d(n^{d-1})$, \cref{equ::lambdaLogSum}, and the definition that $\lambda(z)=1$ for $z\in D \cup E^\ast$, we have
\[
\log_2\bigg(
\prod_{z\in U}
\frac{2}{\lambda(z)(1-2^{-\theta})}
\bigg)
=
|U|\log_2\frac{2}{1-2^{-\theta}}
+
\sum_{z\in C^+}\log_2\frac{1}{\lambda_+(z)}
=
O_{d,\eta}(n^{d-1}),
\]
Thus,
\[
|\F_U|
 \le
2^{6\theta R+O_{d,\eta}(n^{d-1})}
=
2^{\eta R+O_{d,\eta}(n^{d-1})}.
\]
This completes the proof of \cref{lem::upperProfileCount}.
\end{proof}

\section{Concluding remarks}
\label{sec::concluding}

Together with the previous results for $d=1$ and $d=2$, \cref{thm::main} determines the total number of sum-free subsets of $[n]^d$ up to a boundary-order error in the exponent for every fixed dimension. Our proof does not give an upper bound for the number of sum-free sets of a given size or a stability theorem for near-extremal sum-free sets. It seems plausible that a stability result could be obtained by examining the proof of Keevash and Lim~\cite{keevash2026largest} more carefully and tracing the near-equality cases in their dual weights. Indeed, if a sum-free set has size $M([n]^d)-o(n^d)$, then the total slack must be small. Then, one would need to convert this averaged near-equality into structural information on most fibers and show that the set is close to the optimal stripe.

The repeated-source term handled by \cref{prop::localEntropy}~(2) can also be handled by ordinary fractional Shearer together with the Green--Morris theorem~\cite[Theorem 1.1]{green2016counting}. Ordinary Shearer uses the blocks $\{v,2v\}$ and $\{v\}$ and therefore requires an estimate for $H(X,Z)+H(X)$. After fixing the minimum of the source and applying the log-sum inequality, this reduces to a weighted sum with terms of the form $2^{-| (S+S) \cap [0,t] |/2}$. Splitting according to the size of $S$ and the size of $S+S$, and applying the Green--Morris estimate to the dense small-doubling sets, gives the required strict exponential saving. This was the argument used in an earlier version of the paper.

A second alternative is a modulo-$3$ refinement. Split every fiber into the three residue classes and record the corresponding residue-wise profile values. For each weighted triple, use the six ordered pairs of distinct source residues, each with half of the original coefficient, and use the target residue equal to the sum of the source residues modulo $3$. Even when the two base points are equal, the two source pieces are disjoint. Across the six blocks, every source residue has the required total multiplicity two and every target residue has multiplicity one. Thus, only the ordinary three-fiber estimate is needed. This refinement requires us to redefine the profiles, defects, and cancellation equalities.

A third possible route is an anti-collision refinement of the additive couplings in the Keevash--Lim construction. The estimates needed for this approach are
\[
\sum_{\substack{(x,y,z) \in \T\\ y-x=q}}
\o_1(x,y,z)=O_d(n^{d-2})
\qquad
\text{for every }q \in \{0,\pm e_1,\pm2e_1\},
\]
and
\[
\sum_{\substack{(x,y,z) \in \T\\ x=y}}
\o_3(x,y,z)=O_d(n^{d-2}).
\]
A preliminary check of the recursive coupling construction of Keevash and Lim suggests that such estimates should be obtainable, and if so, one may delete the weights over all such triples. Since the deleted mass is $O_d(n^{d-2})$, these terms contribute only $O_d(n^{d-1})$ to the exponent. The resulting loss in $W$ can be repaired by adding singleton blocks with coefficients equal to the lost coordinate weights. Every remaining weighted triple has distinct source fibers and is handled by \cref{prop::localEntropy}~(1). However, establishing these estimates seems to require a substantially longer refinement of the recursive coupling construction of Keevash and Lim. In particular, we cannot just use their Corollary~2.4 as a black box and take the constructions in their Lemma~4.3.

The author prefers the strong fractional entropy argument used in the current proof. It only requires $H(X,Z)+H(X\mid Z)$, reduces the repeated-source estimate to the distinct-source estimate through Zhao's bipartite swapping lemma, and uses the Keevash--Lim couplings only through the support and marginal properties recorded in \cref{sec::keevashlim}.

\section*{Acknowledgments and declaration on the use of generative AI}

This work was developed with the assistance of ChatGPT. The initial idea of combining the dual weights of Keevash and Lim with defects of profiles was proposed by the author. Subsequent proof development, including the completion and simplification of several technical arguments, was carried out through a series of iterative discussions with the model. In particular, the anti-collision approach described in \cref{sec::concluding} was the author's original proposed method for handling repeated source fibers. When the author first asked ChatGPT to pursue this approach, the model did not produce a proof and noted that the problem appeared potentially difficult. After further discussions, ChatGPT suggested the modulo-$3$ alternative, and a plausible proof was eventually developed through further work, but the author regarded the modulo-$3$ approach as artificial and preferred not to use it. The model then suggested the Green--Morris theorem, which was used in an earlier version of the paper. Later, the author observed that the repeated-source block should require only a conditional entropy term and asked whether an ordered form of Shearer's inequality could avoid the Green--Morris input. The model located the strong fractional entropy inequality of Madiman and Tetali and suggested combining it with Zhao's bipartite swapping lemma. This yielded the repeated-source proof used in the final version. The anti-collision approach was also revisited before the submission; this time ChatGPT gave a plausible proof, and a preliminary inspection by the author suggests that the proposed proof may be correct, but the author has not verified it.

The author independently checked all mathematical statements, proofs, and calculations and takes full responsibility for all contents of this paper.

The author is grateful to J\'ozsef Balogh, Guanghui Wang, and Wenling Zhou for helpful discussions.

\appendix
\crefalias{section}{appendix}

\section{Coordinate sums and translations of the Keevash--Lim weights}
\label{app::KLCoordinateSums}

The purpose of this appendix is to prove \cref{lem::KLWeights,lem::w1Weights,lem::w2Weights,lem::w3Weights,lem::translationVariation}. We verify the support refinements and compute the coordinate sums. We use these formulas to show that a bounded translation changes the covering function $W$ by only $O_d(n^{d-2})$.

We use two parts of the work of Keevash and Lim. Corollary~2.4 in~\cite{keevash2026largest} gives the uniform additive couplings. In the proof of Lemma~4.3 in~\cite{keevash2026largest}, they construct three weight components $w_1,w_2,w_3$ using these couplings. We take exactly these components. We do not modify the weights and only verify the properties we need. To compare the notation, their $[n+1]$ is our $[0,n]$, their $\mathbf 1^Tv$ is our $h(v)$, and their $\mathcal A,\mathcal B,\mathcal C,\mathcal D,\mathcal E,w_i,W_i$ correspond to our $A,B,C,D,E,\o_i,W_{\o_i}$. Recall that
\begin{align*}
A \ce \{v : 0 \le h&(v) < m-n\}, \,\,
B \ce \{v : m-n \le h(v) \le m\}, \,\,
C \ce \{v : m < h(v) < 2m-n\},\\
D& \ce \{v : 2m-n \le h(v) \le 2m\}, \quad
E \ce \{v : 2m < h(v) \le (d-1)n\}.
\end{align*}
and
\[
C^- \ce \left\{v : m < h(v) \le \left\lfloor {(3m-n-1)}/{2}\right\rfloor\right\},
\quad
C^+ \ce \left\{v : \left\lfloor {(3m-n-1)}/{2}\right\rfloor < h(v) < 2m-n\right\}.
\]

For every positive integer $j$ and every integer $k$, let $L_j(k) \ce \{v \in [0,n]^j : h(v)=k\}$ and $\l_j(k) \ce |L_j(k)|$.

For convenience, we restate Corollary~2.4 in~\cite{keevash2026largest} in our notation.

\begin{lemma}[Keevash--Lim]\label{lem::KLAdditiveCoupling}
For every positive integer $j$ and all non-negative integers $r,s$ with $r+s \le jn$, there is a coupling $(X,Y,Z)$ on $L_j(r)\times L_j(s)\times L_j(r+s)$ such that $X+Y=Z$ and $X,Y,Z$ are uniformly distributed on $L_j(r),L_j(s),L_j(r+s)$, respectively.
\end{lemma}

Let $f_d$ denote the density of the sum of $d$ independent uniform random variables on $[0,1]_{\Rs}$. Keevash and Lim showed in Section~3 of~\cite{keevash2026largest} that $u_d^*$ is the unique solution in $(0,d)_\Rs$ of $f_d(u)=2f_d(2u)$.
As mentioned in \cref{subsec::fiberAndRegions}, $d/3 < u_d^* < d/2$. We include a short explanation here. Since $f_d$ is symmetric with respect to $d/2$, we have $f_d(d/3)=f_d(2d/3)$, and hence $f_d(d/3)-2f_d(2d/3)=-f_d(d/3)<0$. On the other hand, $f_d(d/2)-2f_d(d)=f_d(d/2)>0$. By continuity and the uniqueness of the solution in $(0,d/2)_{\Rs}$, it follows that $d/3 < u_d^* < d/2$. Recall that $m \ce \lfloor u_d^*n\rfloor$.

\begin{proof}[Proof of \cref{lem::KLWeights}]
The existence, nonnegativity, original support statements, the bound $W_{\o} \le 1$ on $C$, and the covering estimate outside $C$ are given by Lemma~4.3 in~\cite{keevash2026largest}. We just need to prove the sharper support descriptions $\supp(\o_2) \subseteq \T(A,C^-,C^+)$ and $\supp(\o_3) \subseteq \T(C^-,C^-,E)$.

In the construction of $w_2$ in the proof of Lemma~4.3 in~\cite{keevash2026largest}, they let
\[
1 \le r<m-n-1, \qquad
s = s(r) \ce \left\lfloor\frac{3m-n-r}{2}\right\rfloor, \qquad
t = t(r) \ce r+s.
\]
The $h$-values of the first source, second source, and target are $r,s,t$, respectively. Since $1 \le r<m-n-1$, the first source lies in $A$. As $r$ varies through this range, we have
\[
m+1 \le s \le \left\lfloor\frac{3m-n-1}{2}\right\rfloor,
\qquad
\left\lceil\frac{3m-n}{2}\right\rceil \le t \le 2m-n-1.
\]
Hence, the second source lies in $C^-$ and the target lies in $C^+$.

In the construction of $w_3$ in the proof of Lemma~4.3 in~\cite{keevash2026largest}, they let
\[
2m+2 \le t \le (d-1)n, \qquad
r = r(t) \ce \left\lfloor\frac{t}{2}\right\rfloor, \qquad
s \ce t-r=\left\lceil\frac{t}{2}\right\rceil.
\]
The $h$-values of the two sources and the target are $r,s,t$, respectively. Since $3m-dn=(3u_d^*-d)n+O(1) \ge 3$ for every sufficiently large $n$, we have
\[
m+1 \le r \le s \le \left\lceil\frac{(d-1)n}{2}\right\rceil
 \le \left\lfloor\frac{3m-n-1}{2}\right\rfloor.
\]
Therefore, both sources lie in $C^-$.
\end{proof}

\begin{proof}[Proof of \cref{lem::w1Weights}]
In the construction of $w_1$ in the proof of Lemma~4.3 in~\cite{keevash2026largest}, Keevash and Lim use the coupling $(X_1,Y_1,Z_1)$ supplied by \cref{lem::KLAdditiveCoupling} with $j=d$ and $r=s=m$. Hence, $X_1+Y_1=Z_1$, and $X_1,Y_1,Z_1$ are uniform on $L_d(m),L_d(m),L_d(2m)$, respectively. Let $\pi$ be the projection that removes the last coordinate. The definition of $\o_1$ by Keevash and Lim is
\[
\o_1(x,y,z) \ce |D|\cdot \Pr\bigl((\pi(X_1),\pi(Y_1),\pi(Z_1))=(x,y,z)\bigr).
\]
Note that $\pi$ gives bijections $L_d(m)\to B$ and $L_d(2m)\to D$, as $B=\{v\in V:m-n \le h(v) \le m\}$ and $D=\{v\in V:2m-n \le h(v) \le 2m\}$. Therefore, $\pi(X_1),\pi(Y_1)$ are uniform on $B$ and $\pi(Z_1)$ is uniform on $D$; also, we get $|B|=\l_d(m)$ and $|D|=\l_d(2m)$.

For every $v\in B$, we have
\begin{equation} \label{equ::w1xsum}
\sum_{y,z\in V}\o_1(v,y,z)=|D|\cdot\sum_{y,z\in V}\Pr\bigl((\pi(X_1),\pi(Y_1),\pi(Z_1))=(v,y,z)\bigr)=|D|\cdot\Pr\bigl(\pi(X_1)=v\bigr)=\frac{|D|}{|B|}.
\end{equation}
Similarly, we have
\begin{equation} \label{equ::w1ysum}
\sum_{x,z\in V}\o_1(x,v,z)=|D|/|B|.
\end{equation}
By the definition of $m$, we have $m = u_d^*n + O(1)$. The lattice-slice estimates used in the construction of $w_1$ in the proof of Lemma~4.3 in~\cite{keevash2026largest} therefore give $|B| = f_d(u_d^*)n^{d-1} + O_d(n^{d-2})$ and $|D| = f_d(2u_d^*)n^{d-1} + O_d(n^{d-2})$. Since $f_d(u_d^*) = 2f_d(2u_d^*)$, we have $|B| = 2|D| + O_d(n^{d-2})$ and $|B| = \Theta_d(n^{d-1})$. Therefore, $\frac{|D|}{|B|} = \frac{1}{2} + O_d(n^{-1})$. This proves \cref{equ::w1CoordinateSums}.

For every $v\in D$, we have
\begin{equation} \label{equ::w1zsum}
\sum_{x,y\in V}\o_1(x,y,v)=|D|\cdot\sum_{x,y\in V}\Pr\bigl((\pi(X_1),\pi(Y_1),\pi(Z_1))=(x,y,v)\bigr)=|D|\cdot\Pr\bigl(\pi(Z_1)=v\bigr)=1.
\end{equation}
Finally,
\[
\sum_{x,y,z\in V}\o_1(x,y,z)=|D|\cdot\sum_{x,y,z\in V}\Pr\bigl((\pi(X_1),\pi(Y_1),\pi(Z_1))=(x,y,z)\bigr)=|D|=O_d(n^{d-1}). \qedhere
\]
\end{proof}

For the proof of \cref{lem::w2Weights,lem::w3Weights,lem::translationVariation}, we need the following lemma.
\begin{lemma}\label{lem::sliceDifference}
For every non-negative integer $k$ and every integer $r \in [1,4]$, we have
\begin{equation} \label{equ::sliceDifference}
|\l_{d-1}(k+r)-\l_{d-1}(k)|=O_d(n^{d-3}).
\end{equation}
Moreover, $\l_{d-1}(k)=\Theta_d(n^{d-2})$ for every integer $k \in [n,2m-n-1]$.
\end{lemma}

\begin{proof}
We first prove \cref{equ::sliceDifference}. By considering the first coordinate, we have $\l_{d-1}(k)=\sum_{a=0}^{n}\l_{d-2}(k-a)$, and hence
\[
\l_{d-1}(k+1)-\l_{d-1}(k)=\l_{d-2}(k+1)-\l_{d-2}(k-n).
\]
Note that $\l_{d-2}(s) \le (n+1)^{d-3}$ for every integer $s$, since fixing the first $d-3$ coordinates determines the last coordinate. Hence, $|\l_{d-1}(k+1)-\l_{d-1}(k)|=O_d(n^{d-3})$. Summing at most four consecutive differences proves \cref{equ::sliceDifference}.

We next prove the second assertion. The upper bound is immediate, since fixing the first $d-2$ coordinates determines the last coordinate. Fix $k\in[n,2m-n-1]$. Since $m=\lfloor u_d^*n\rfloor$ and $u_d^*<d/2$, we have $1/(d-1) \le k/((d-1)n)<(2u_d^*-1)/(d-1)<1$. Therefore, there is a constant $b_d>0$ such that $k/(d-1)$ has distance at least $b_dn$ from both $0$ and $n$. Choose each of $v_1,\ldots,v_{d-2}$ among the integers satisfying $|v_i-k/(d-1)| \le b_dn/(2(d-2))$, and let $v_{d-1} \ce k-\sum_{i=1}^{d-2}v_i$. There are $\Omega_d(n^{d-2})$ such choices. Moreover, $|v_{d-1}-k/(d-1)| \le b_dn/2$, so all the coordinates lie in $[0,n]$. Hence, $\l_{d-1}(k)=\Omega_d(n^{d-2})$. Together with the upper bound, this proves the second assertion.
\end{proof}

\begin{proof}[Proof of \cref{lem::w2Weights}]
The construction of $w_2$ in the proof of Lemma~4.3 in~\cite{keevash2026largest} goes as follows. For every integer $r$ with $1 \le r<m-n-1$, let $s = s(r) \ce \left\lfloor\frac{3m-n-r}{2}\right\rfloor$ and $t \ce r+s$. For every $r$, by \cref{lem::KLAdditiveCoupling}, there is a coupling $(X_2^{(r)},Y_2^{(r)},Z_2^{(r)})$ on $L_{d-1}(r) \times L_{d-1}(s) \times L_{d-1}(t)$ with $X_2^{(r)}+Y_2^{(r)}=Z_2^{(r)}$, where $X_2^{(r)},Y_2^{(r)},Z_2^{(r)}$ are uniform on $L_{d-1}(r),L_{d-1}(s),L_{d-1}(t)$, respectively. For every $x$ with $1 \le h(x)<m-n-1$, the definition of $\o_2$ is
\[
\o_2(x,y,z) \ce \l_{d-1}(h(x)) \cdot \Pr \Big(
\big(X_2^{(h(x))},Y_2^{(h(x))},Z_2^{(h(x))}\big)=(x,y,z)
\Big),
\]
and $\o_2(x,y,z) \ce 0$ when $h(x) \notin [1,m-n-1)$.

Now, for every $v \in A$ with $h(v) \in [1,m-n-2]$, we have
\[
\sum_{y,z \in V}\o_2(v,y,z)=\l_{d-1}(h(v)) \cdot \Pr\bigl(X_2^{(h(v))}=v\bigr)=1.
\]
If $h(v)$ does not belong to this interval, then by definition, no component has $v$ as its first source, so the coordinate sum is $0$.

Fix $v \in C^-$. By the definition of $\o_2$, the value $\o_2(x,v,z)$ can be non-zero only when the second source layer of the component indexed by $h(x)$ is $L_{d-1}(h(v))$. Equivalently, $\lfloor(3m-n-h(x))/2\rfloor=h(v)$, which holds precisely when $h(x)\in\{3m-n-2h(v)-1,3m-n-2h(v)\}$. For either admissible value $r=h(x)$, the total contribution of the component indexed by $r$ to $\lambda_-(v)$ is $\l_{d-1}(r)\Pr(Y_2^{(r)}=v)=\l_{d-1}(r)/\l_{d-1}(h(v))$. Thus, $\lambda_-(v)$ depends only on $h(v)$. Moreover, since $v \in C^-$, we have $h(v) > m$, so $1 \le 3m-n-2h(v) \le m-n-2$ and hence at least one of the two candidate first-source layers belongs to $[1,m-n-1)$. Hence, $\lambda_-(v)>0$. We also have $\lambda_-(v) \le 1$, since $\lambda_-(v) \le W(v)$ by definition and $W(v) \le 1$ by \cref{lem::KLWeights}.

Similarly, fix $v' \in C^+$. The value $\o_2(x,y,v')$ can be non-zero only when the target layer of the component indexed by $h(x)$ is $L_{d-1}(h(v'))$. Since $t=\lfloor(3m-n+r)/2\rfloor$, this holds precisely when $h(x)\in\{2h(v')-(3m-n),2h(v')-(3m-n)+1\}$. For either admissible value $r=h(x)$, the total contribution of the component indexed by $r$ to $\lambda_+(v')$ is $\l_{d-1}(r)\Pr(Z_2^{(r)}=v')=\l_{d-1}(r)/\l_{d-1}(h(v'))$. Thus, $\lambda_+(v')$ depends only on $h(v')$. Moreover, since $v'\in C^+$, we have $0 \le 2h(v')-(3m-n) \le m-n-2$. If $2h(v')-(3m-n)=0$, then the second candidate is $1$; otherwise, the first candidate belongs to $[1,m-n-2]$. Hence, at least one of the two candidates is admissible, so $\lambda_+(v')>0$. We also have $\lambda_+(v') \le 1$, since $\lambda_+(v') \le W(v')$ by definition and $W(v') \le 1$ by \cref{lem::KLWeights}.

Now, we prove \cref{equ::lambdaLogSum} by considering the lower bound on $\lambda_-(v)$ and $\lambda_+(v')$ more carefully. For each $h$-level in $C^-$ or $C^+$, as proved above, we can choose one admissible first-source level $r\in[1,m-n-2]$ from the two candidates found above. Since the relevant coordinate sum is the sum of non-negative contributions, it is at least the contribution from this chosen component. Note that there is a constant $c_d>0$ such that $\l_{d-1}(r) \ge c_d\min\{r^{d-2},n^{d-2}\}$. Indeed, for $r \le n$, we have $\l_{d-1}(r)=\b{r+d-2}{d-2}=\Omega_d(r^{d-2})$; if $r \ge n$, then $r \le m-n-2<2m-n-1$, so $r\in[n,2m-n-1]$ and the second assertion of \cref{lem::sliceDifference} gives $\l_{d-1}(r)=\Theta_d(n^{d-2})$. Since every layer in $[0,n]^{d-1}$ has at most $(n+1)^{d-2}$ points, after decreasing $c_d$ if necessary, the corresponding value of $\lambda_-$ or $\lambda_+$ is at least $c_d\min\{1,(r/n)^{d-2}\}$. If $r \ge n$, then the corresponding logarithmic contribution is at most $\log_2(2+c_d^{-1})=O_d(1)$. If $r<n$, then the reciprocal of the corresponding coordinate sum is at most $c_d^{-1}(n/r)^{d-2}$ and hence, its logarithmic contribution is at most
\[
\log_2\left(\left(2+\frac{1}{c_d}\right)
\left(\frac{n}{r}\right)^{d-2}\right)
=
\log_2\left(2+\frac{1}{c_d}\right)
+
(d-2)\log_2\frac{n}{r}.
\]
Therefore, in both cases, the logarithmic contribution of every point on this level is at most
\[
C_d\left(1+\max\left\{0,\log_2\frac{n}{r}\right\}\right).
\]
Note that each $r$ can correspond to at most one $h$-level in $C^-$ and at most one $h$-level in $C^+$. Indeed, when the $h$-level increases by one, both candidate first-source levels change by two, so the two-element candidate sets for distinct $h$-levels are disjoint. Moreover, each $h$-level contains at most $(n+1)^{d-2}$ points. Consequently, the left-hand side of \cref{equ::lambdaLogSum} is at most
\[
C_d(n+1)^{d-2}
\sum_{r=1}^{m-n-2}
\left(
1+\max\left\{0,\log_2\frac{n}{r}\right\}
\right)
=
O_d(n^{d-1}),
\]
since $m=O_d(n)$ and $\sum_{i=1}^{n}\log_2(n/i)=O(n)$. This proves \cref{equ::lambdaLogSum}.

Finally,
\[
\sum_{x,y,z \in V}\o_2(x,y,z)=\sum_{r=1}^{m-n-2}\l_{d-1}(r)
 \le dn \cdot (n+1)^{d-2}
=O_d(n^{d-1}). \qedhere
\]
\end{proof}

\begin{proof}[Proof of \cref{lem::w3Weights}]
The construction of $\o_3$ in the proof of Lemma~4.3 in~\cite{keevash2026largest} goes as follows. For $d=3$, there is nothing to construct. Assume $d \ge 4$. For every $t \in [2m+2, (d-1)n]$, let $r = r(t) = \left \lfloor \frac{t}{2} \right \rfloor$ and $s = t - r = \left \lceil \frac{t}{2} \right \rceil$. By \cref{lem::KLAdditiveCoupling}, for each $t$, there is a coupling $(X_3^{(t)}, Y_3^{(t)}, Z_3^{(t)})$ on $L_{d-1}(r) \times L_{d-1}(s) \times L_{d-1}(t)$ with $X_3^{(t)} + Y_3^{(t)} =Z_3^{(t)}$ such that $X_3^{(t)}, Y_3^{(t)}, Z_3^{(t)}$ are uniform on $L_{d-1}(r)$, $L_{d-1}(s)$, $L_{d-1}(t)$, respectively. For every $z$ with $2m+2 \le h(z) \le (d-1)n$, the definition of $\o_3$ is
\[
\o_3(x,y,z)
 \ce
\l_{d-1}(h(z))\cdot
\Pr\Big(
\big(X_3^{(h(z))},Y_3^{(h(z))},Z_3^{(h(z))}\big)
=(x,y,z)
\Big),
\]
and $\o_3(x,y,z) \ce 0$ otherwise.

For every $v \in E$ with $h(v) \ge 2m+2$, we have
\[
\sum_{x,y \in V}\o_3(x,y,v)=\l_{d-1}(h(v)) \cdot \Pr\bigl(Z_3^{(h(v))}=v\bigr)=1.
\]
For $v$ with $h(v) = 2m+1$, then by definition, no component has $v$ as its target, so the coordinate sum is $0$.

Finally,
\[
\sum_{x,y,z \in V}\o_3(x,y,z)=\sum_{t=2m+2}^{(d-1)n}\l_{d-1}(t)
 \le dn \cdot (n+1)^{d-2}
=O_d(n^{d-1}). \qedhere
\]
\end{proof}

\begin{proof}[Proof of \cref{lem::translationVariation}]
For a region $X \subseteq V$ and a vector $u \in \Zs^{d-1}$, let $\partial_uX \ce \{v \in X:v+u \notin X\}$. By the definition of $\R$, every vector among $r,s,r+s$ has bounded coordinates, so we have $|\partial_uX|=O_d(n^{d-2})$ for every $X \in \{A,B,C^-,C^+,D\}$ and every $u \in \{r,s,r+s\}$. Note that
\begin{align*}
\B_1^{r,s} & \subseteq \{(x,y,z) \in \supp(\o_1):x \in \partial_r B\textrm{ or }y \in \partial_s B\textrm{ or }z \in \partial_{r+s}D\}, \\
\B_2^{r,s} & \subseteq \{(x,y,z) \in \supp(\o_2):x \in \partial_rA\textrm{ or }y \in \partial_sC^-\textrm{ or }z \in \partial_{r+s}C^+\}.
\end{align*}
By \cref{equ::w1xsum,equ::w1ysum,equ::w1zsum}, every coordinate sum of $\o_1$ is at most $1$. Therefore, we have
\[
\sum_{(x,y,z) \in \B_1^{r,s}}\o_1(x,y,z)
 \le
|\partial_rB|+|\partial_sB|+|\partial_{r+s}D|
=
O_d(n^{d-2}).
\]
Similarly, every coordinate sum of $\o_2$ is at most $1$ by \cref{lem::w2Weights}. Hence,
\[
\sum_{(x,y,z) \in \B_2^{r,s}}\o_2(x,y,z)
 \le
|\partial_rA|+|\partial_sC^-|+|\partial_{r+s}C^+|
=
O_d(n^{d-2}).
\]
This proves the first assertion.

For \cref{equ::translationVariation}, let
\[
p_1(v) \ce \sum_{y,z \in V}\o_i(v,y,z),\qquad
p_2(v) \ce \sum_{x,z \in V}\o_i(x,v,z),\qquad
p_3(v) \ce \sum_{x,y \in V}\o_i(x,y,v).
\]
We extend each $p_j$ by zero outside $V$. It suffices to prove that each of $\sum_{v \in V}|p_1(v-r)-p_1(v)|$, $\sum_{v \in V}|p_2(v-s)-p_2(v)|$, and $\sum_{v \in V}|p_3(v-r-s)-p_3(v)|$ is $O_d(n^{d-2})$. Indeed, before imposing the regional restrictions, the first, second, and third coordinate sums of the translated weight at $v$ are $p_1(v-r)$, $p_2(v-s)$, and $p_3(v-r-s)$, respectively. Passing to $\o_i^{r,s}$ deletes total weight $\sum_{(x,y,z)\in \B_i^{r,s}}\o_i(x,y,z)$. Therefore, $\sum_{v\in V}|W'(v)-W(v)|$ is at most the sum of the preceding three terms plus $3\sum_{(x,y,z)\in\B_i^{r,s}}\o_i(x,y,z)$, which is $O_d(n^{d-2})$ by the first assertion.

Suppose first that $i=1$. By \cref{equ::w1xsum,equ::w1ysum,equ::w1zsum}, we have that $p_1$ and $p_2$ are equal to $\frac{1}{2}+O_d(n^{-1})$ on $B$ and are $0$ outside $B$, while $p_3$ is equal to $1$ on $D$ and is $0$ outside $D$. Since $r,s,r+s$ have bounded coordinates, translating $B$ or $D$ changes membership only for $O_d(n^{d-2})$ points. Also $0 \le p_j \le 1$. This proves all three estimates.

It remains to consider the case $i=2$.

For $p_1$, by \cref{lem::w2Weights}, the first coordinate sum is the indicator function of the set $\{v \in V:1 \le h(v) \le m-n-2\}$, so the same argument gives $\sum_{v \in V}|p_1(v-r)-p_1(v)|=O_d(n^{d-2})$.

For $p_2$, the main idea is that $p_2(v)$ is the ratio of the sizes of layers, and a bounded translation changes all the layer indices by also just a bounded value. By \cref{lem::w2Weights}, we have $p_2(v)=0$ for $v\notin C^-$. For $v\in C^-$, the only first-source contributing to $p_2(v)$ are those with $h$-value $3m-n-2h(v)-1$ and $3m-n-2h(v)$, and the contribution of each admissible level is its layer size divided by $\l_{d-1}(h(v))$. Except on $O_d(n^{d-2})$ points near the boundaries of $C^-$ or of $V$, both $v$ and $v-s$ lie in the same region. If they lie outside $C^-$, then $p_2(v)=p_2(v-s)=0$. Suppose that both lie in $C^-$. Then, by the definition of $C^-$, we have $m+1 \le h(v), h(v-s) \le \left\lfloor {(3m-n-1)}/{2}\right\rfloor$, so $1 \le 3m-n-2h(v),\,3m-n-2h(v-s) \le m-n-2$ and hence they are in $A$, and the first candidate is inadmissible only when the second candidate equals $1$. Excluding these at most two $h$-levels removes only $O_d(n^{d-2})$ additional points. For every remaining point, both numerators contain two corresponding layer sizes. Since $h(v)-h(v-s)=h(s)$ and $|h(s)| \le 2$, the denominator indices differ by at most $2$, while the corresponding numerator indices differ by at most $4$. For every such vertex, write
\[
p_2(v)=\frac{N}{\l_{d-1}(h(v))}
\qquad\textrm{and}\qquad
p_2(v-s)=\frac{N'}{\l_{d-1}(h(v-s))}.
\]
By the first assertion of \cref{lem::sliceDifference}, we have
\[
|N-N'|
+
\bigl|
\l_{d-1}(h(v))-\l_{d-1}(h(v-s))
\bigr|
=
O_d(n^{d-3}).
\]
Also, since $v, v-s \in C$, both denominator indices lie in $[m+1, 2m-n-1] \subseteq [n,2m-n-1]$. Hence, the second assertion of \cref{lem::sliceDifference} gives $\l_{d-1}(h(v)),\,\l_{d-1}(h(v-s)) = \Theta_d(n^{d-2})$. Moreover, $N'=O_d(n^{d-2})$, as it is the sum of the sizes of two layers in $[0,n]^{d-1}$. Therefore, we have
\begin{align*}
\left|p_2(v)-p_2(v-s)\right| & =
\bigg|\Big(\frac{N}{\l_{d-1}(h(v))} - \frac{N'}{\l_{d-1}(h(v))}\Big) + \Big(\frac{N'}{\l_{d-1}(h(v))} -\frac{N'}{\l_{d-1}(h(v-s))}\Big) \bigg| \\
& \le \frac{|N-N'|}{\l_{d-1}(h(v))} +\frac{N'\bigl|\l_{d-1}(h(v))-\l_{d-1}(h(v-s))\bigr|}{\l_{d-1}(h(v))\l_{d-1}(h(v-s))}
=
O_d(n^{-1}).
\end{align*}
Summing over the $O_d(n^{d-1})$ points and using $0 \le p_2 \le 1$ on the excluded $O_d(n^{d-2})$ points gives
\[
\sum_{v\in V}|p_2(v-s)-p_2(v)|=O_d(n^{d-2}).
\]

For $p_3$, the idea is similar. Except on $O_d(n^{d-2})$ points near the boundaries of $C^+$ or of $V$, both $v$ and $v-r-s$ lie in the same region. If they lie outside $C^+$, then $p_3(v)=p_3(v-r-s)=0$. Suppose that both lie in $C^+$. By the proof of \cref{lem::w2Weights}, the only first-source levels contributing to $p_3(v)$ are $2h(v)-(3m-n)$ and $2h(v)-(3m-n)+1$, and the contribution of each admissible candidate is its layer size divided by $\l_{d-1}(h(v))$. The corresponding candidates for $p_3(v-r-s)$ are $2h(v-r-s)-(3m-n)$ and $2h(v-r-s)-(3m-n)+1$. After excluding the endpoint cases as above, corresponding candidates are admissible simultaneously. Their indices differ by at most $4$, while the denominator indices differ by at most $2$. The same calculation as for $p_2$ gives
\[
\sum_{v \in V}|p_3(v-r-s)-p_3(v)|=O_d(n^{d-2}).
\]

For \cref{equ::translatedWeightError}, for every $v \in C$, we have $(W'(v)-1)_+ \le |W'(v)-W(v)|$ because $W(v) \le 1$; for every $v \in V \sm C$, we have $|W'(v)-1| \le |W(v)-1|+|W'(v)-W(v)|$. Then, by \cref{lem::KLWeights,equ::translationVariation}, we have
\begin{align*}
 &\sum_{v \in C}(W'(v)-1)_+
 +\sum_{v \in V \sm C}|W'(v)-1| \\
 \le &
 \sum_{v \in C}|W'(v)-W(v)|
 +\sum_{v \in V \sm C}
 \bigl(|W(v)-1|+|W'(v)-W(v)|\bigr) \\
 = &
 \sum_{v \in V}|W'(v)-W(v)|
 +\sum_{v \in V \sm C}|W(v)-1|
 =
 O_d(n^{d-2}). \qedhere
\end{align*}
\end{proof}

\section{Proof of Ghosal's lemma}\label{app::GhosalLem}

\begin{proof}[Proof of \cref{lem::Ghosal}]
The proof follows the main idea from Ghosal's argument of considering whether a certain interval is contained entirely in $S_1+S_2$ or not. We slightly simplify the argument.

Recall that $\S_t \ce \{(S_1,S_2): S_1, S_2 \subseteq [0,t], \, 0 \in S_1 \cap S_2, \,t+1 \notin S_1+S_2\}$ and
\[
g(t) \ce
\sum_{(S_1,S_2)\in\S_t}
2^{-| (S_1+S_2) \cap [0,t] |},
\]
and our target is to prove that there are absolute constants $K>0$ and $\eta>0$ such that
\[
g(t) \le K(2-\eta)^t.
\]
We may assume that $K$ is sufficiently large and $\eta$ is sufficiently small, and also $t$ is sufficiently large. Let
\[
J_t \ce \left[\frac{3t}{4}+1,t\right].
\]
Let $\S_t^0 \ce \{(S_1,S_2) \in \S_t: J_t \subseteq S_1+S_2\}$ and $\S_t^k \ce \{(S_1,S_2) \in \S_t: t+1-k \notin S_1 +S_2\}$ for every $k \in [1,t/4]$. For every $i \in [0,t/4]$, define
\[
g_i(t) \ce
\sum_{(S_1,S_2) \in \S_t^i} 2^{-| (S_1+S_2) \cap [0,t] |}.
\]
Then,
\begin{equation} \label{equ::gtlesumg0tgkt}
    g(t) \le g_0(t) + \sum_{k: 1 \le k \le t/4} g_k(t).
\end{equation}
We will bound the values of $g_0(t)$ and $g_k(t)$ separately.

\begin{claim} \label{cla::g0tle}
There are absolute constants $K_0,\eta_0>0$ such that for every integer $t \ge 0$,
\[
g_0(t) \le K_0(2-\eta_0)^t.
\]
\end{claim}
\begin{proof}
For every $(S_1,S_2)\in\S_t^0$, since $0\in S_1 \cap S_2$, we have $S_1 \cup S_2 \subseteq S_1 + S_2$. Then, combining with $J_t \subseteq S_1+S_2$, we have
\[
S_1 \cup S_2 \cup J_t
 \subseteq
(S_1+S_2) \cap [0,t].
\]
Therefore,
\[
g_0(t)
=
\sum_{(S_1,S_2)\in\S_t^0}
2^{-| (S_1+S_2) \cap [0,t] |}
 \le
\sum_{(S_1,S_2)\in\S_t^0}
2^{-|S_1 \cup S_2 \cup J_t|}
 \le
\sum_{(S_1,S_2)\in\S_t}
2^{-|S_1 \cup S_2 \cup J_t|}.
\]

First, consider the case where $t$ is even. We partition $[t]$ into the pairs
\[
D_t = \Big\{
\{u,\,t+1-u\}:
1 \le u \le \frac{t}{2}
\Big\}
.
\]
Note that for every $\{u,v\} \in D_t$ and $(S_1,S_2) \in \S_t$, we have that if $u \in S_1$, then $v \notin S_2$ and if $u \in S_2$, then $v \notin S_1$. Also, the choices on distinct pairs in $D_t$ are independent. Therefore,
\[
\sum_{(S_1,S_2)\in\S_t}
2^{-|S_1 \cup S_2 \cup J_t|}
=
\frac{1}{2}
\prod_{\{u,v\}\in D_t}
\Bigg(
\sum_{\substack{T_1,T_2 \subseteq \{u,v\}\\
t+1\notin T_1+T_2}}
2^{-|T_1 \cup T_2 \cup (J_t \cap \{u,v\})|}
\Bigg),
\]
where the $1/2$ comes from the requirement $0 \in S_1 \cap S_2$. Now we compute the factor corresponding to a fixed pair $\{u,v\}\in D_t$, where $u<v$.

Suppose first that $v\notin J_t$. Among the admissible choices for the membership of $u$ and $v$ in the ordered pair $(S_1,S_2)$, there is one choice in which neither $u$ nor $v$ belongs to $S_1 \cup S_2$, six choices in which exactly one of them belongs to $S_1 \cup S_2$, and two choices in which both belong to $S_1 \cup S_2$. Hence, the total contribution from this pair is
\[
1+6\cdot\frac{1}{2}+2\cdot\frac{1}{4}=\frac{9}{2}.
\]

Now suppose that $v\in J_t$. Then, $v$ belongs to $S_1 \cup S_2 \cup J_t$ regardless of its membership in $S_1$ and $S_2$. If $u\notin S_1 \cup S_2$, then there are four choices for the membership of $v$, giving total contribution $4\cdot\frac{1}{2}$. If $u\in S_1\sm S_2$, then $v$ may belong to neither set or to $S_1$ only, giving total contribution $2\cdot\frac{1}{4}$. The case $u\in S_2\sm S_1$ is symmetric. Finally, if $u\in S_1 \cap S_2$, then there is only one case where $v$ must belong to neither set, giving contribution $1 \cdot 1/4$. Thus, the total contribution from this pair is
\[
 4 \cdot \frac{1}{2}
+2 \cdot \frac{1}{4}
+2 \cdot \frac{1}{4}
+1 \cdot \frac{1}{4}
=
\frac{13}{4}.
\]

Since $|J_t|=\left\lfloor\frac{t}{4}\right\rfloor$, exactly $\lfloor \frac{t}{4}\rfloor$ pairs contribute a factor $13/4$, while the remaining $\frac{t}{2}-\left\lfloor\frac{t}{4}\right\rfloor$ pairs contribute a factor $9/2$. Taking the product over all pairs in $D_t$, we obtain
\[
g_0(t)
 \le
\frac{1}{2} \cdot
\left(\frac{13}{4}\right)^{\lfloor t/4\rfloor}
\left(\frac{9}{2}\right)^{t/2-\lfloor t/4\rfloor}.
\]
When $t$ is odd, we apply the same argument to the $\lfloor t/2\rfloor$ pairs $\{u,t+1-u\}$ with $1 \le u<(t+1)/2$. The remaining element $(t+1)/2$ cannot belong to both $S_1$ and $S_2$, and its three possible states have total contribution
\[
1+\frac{1}{2}+\frac{1}{2}=2.
\]
This cancels the factor $1/2$ contributed by the element $0$. Hence, in both cases,
\[
g_0(t)
 \le
\left(\frac{13}{4}\right)^{\lfloor t/4\rfloor}
\left(\frac{9}{2}\right)^{
\lfloor t/2\rfloor-\lfloor t/4\rfloor}.
\]
Since $\left\lfloor\frac{t}{4}\right\rfloor \le \frac{t}{4}$ and $\left\lfloor\frac{t}{2}\right\rfloor - \left\lfloor\frac{t}{4}\right\rfloor \le \frac{t}{4}+1$, we obtain
\[
g_0(t)
 \le
\frac{9}{2}
\left(\frac{13}{4}\cdot\frac{9}{2}\right)^{t/4}
=
\frac{9}{2}
\left(\frac{117}{8}\right)^{t/4}.
\]
Note that $\left(\frac{117}{8}\right)^{1/4}<2$, since $117/8<16$. Letting $K_0$ be sufficiently large and $\eta_0$ be sufficiently small, we have
\[
g_0(t) \le K_0(2-\eta_0)^t. \qedhere
\]
\end{proof}

For $g_k(t)$, the general idea is similar to the proof of \cref{cla::g0tle}. We introduce the following notation for our discussion. For every integer $k\in[1,t/4]$, let $G_k(t)$ be the graph on vertex set $[t]$, where loops are allowed and $i$ is connected to $j$ if $i+j=t+1$ or $i+j=t+1-k$. For any graph $G$, possibly with loops, let $A \ce \{\es,\{1\},\{2\},\{1,2\}\}$. We say that a function $f:V(G)\to A$ is \emph{good} if, for every edge $uv\in E(G)$, we do not have both $1\in f(u)$ and $2\in f(v)$, and we do not have both $2\in f(u)$ and $1\in f(v)$. In particular, if $u$ has a loop, then $f(u)\ne\{1,2\}$. Let
\[
Z(G) \ce \sum_f2^{-|\{v\in V(G):f(v)\ne\es\}|},
\]
where the sum is over all good functions $f:V(G)\to A$. Note that for every $(S_1,S_2)\in\S_t^k$, if we define $f:[t]\to A$ by $f(i) \ce \{j\in\{1,2\}:i\in S_j\}$, the function $f$ is good on $G_k(t)$. Moreover, the ordered pair $(S_1,S_2)$ is uniquely determined by $f$, since $S_j=\{0\} \cup \{i\in[t]:j\in f(i)\}$ for $j\in\{1,2\}$. Thus, this map is an injection from $\S_t^k$ into the set of good functions on $G_k(t)$.

Since $0\in S_1 \cap S_2$, we have $S_1 \cup S_2 \subseteq (S_1+S_2) \cap [0,t]$, and $|S_1 \cup S_2|=1+|\{i\in[t]:f(i)\ne\es\}|$. Hence,
\begin{equation}\label{equ::gktleZGkt}
g_k(t) \le \frac{1}{2} Z(G_k(t)).
\end{equation}
We consider $Z(G_k(t))$.

\begin{claim}\label{cla::gammaDecomposition}
At most two vertices in $G_k(t)$ contain a loop. Every component containing no loop has at least eight vertices. After the loops of $G_k(t)$ are deleted, every component is a path.
\end{claim}
\begin{proof}
A loop can occur at a vertex $i$ only when $2i=t+1$ or $2i=t+1-k$, so at most two vertices of $G_k(t)$ have a loop.

Let $H$ be the graph obtained from $G_k(t)$ by removing the loops. By definition, every vertex in $H$ has degree at most $2$, so every component is either a cycle or a path. Note that every path of length $2$ in $H$ must have one of the forms
\[
i\mapsto t+1-i\mapsto i-k
\qquad\textrm{or}\qquad
i\mapsto t+1-k-i\mapsto i+k.
\]
Hence, every two steps change the label exactly by $k$ in one fixed direction. In particular, there is no cycle in $H$, and hence every component is a path.

It remains to consider a component containing no loop in $G_k(t)$, which remains the same in $H$. Let $a$ be the smallest element in this component. We first have $a \le k$, since otherwise $a\mapsto t+1-a\mapsto a-k$ is a path in $H$ and $a - k$ is in the same component. Then, this component contains vertices
\[
t+1-a, \ a,\ t+1-k-a,\ a+k,\ t+1-2k-a,\ a+2k,\ t+1-3k-a,\ a+3k.
\]
Indeed, $a+3k \le 4k \le t$ and $t+1-3k-a \ge t+1-4k \ge 1$, so all eight displayed integers belong to $[t]$. Consecutive terms have sum alternately equal to $t+1$ and $t+1-k$, and hence form a path in $H$. Since $H$ contains no cycle, these eight vertices are distinct. Therefore, the component has at least eight vertices.
\end{proof}

Let $P_i$ be the path with $i$ vertices.

\begin{claim}\label{cla::ZvaluePathge8}
There is an absolute constant $\gamma<2$ such that for every integer $i \ge 8$,
\[
Z(P_i) \le \gamma^i.
\]
\end{claim}

\begin{proof}
Write $Z_i \ce Z(P_i)$. For every $i \ge 1$, let $Y_i$ be the contribution to $Z_i$ from the good functions whose value at one fixed endpoint of $P_i$ belongs to $\{\es,\{1\}\}$. By symmetry, the corresponding contribution with endpoint value in $\{\es,\{2\}\}$ is also $Y_i$.

We have $Z_0=1$, $Z_1=1+ 3 \cdot 1/2 = 5/2$, and $Y_1=1+1/2 =3/2$. Conditioning on the value at the final vertex gives, for every $i \ge 2$,
\begin{equation}\label{equ::ZiYiRecurrence}
Z_i=Z_{i-1}+Y_{i-1}+\frac{1}{2}Z_{i-2}
\qquad\textrm{and}\qquad
Y_i=Z_{i-1}+\frac{1}{2}Y_{i-1}.
\end{equation}
Indeed, if the value at the final vertex is $\es$, the contribution is $Z_{i-1}$. If it is $\{1\}$ or $\{2\}$, these two cases together contribute $Y_{i-1}$. If it is $\{1,2\}$, the value at the preceding vertex must be $\es$, and the contribution is $\frac{1}{2}Z_{i-2}$. The recurrence for $Y_i$ follows similarly by considering the final values $\es$ and $\{1\}$.

Iterating \cref{equ::ZiYiRecurrence} gives
\[
Z_8=\frac{32065}{128}<2^8,\qquad
Z_9=\frac{124901}{256}<2^9,\qquad
Y_9=\frac{172565}{512}<\frac{2}{3}\,2^9.
\]
We prove simultaneously that $Z_i<2^i$ for every $i \ge 8$ and $Y_i<\frac{2}{3}\,2^i$ for every $i \ge 9$. For $i \ge 10$, the induction hypotheses and \cref{equ::ZiYiRecurrence} give
\begin{align*}
Z_i
&<
2^{i-1}+\frac{2}{3}\,2^{i-1}+\frac{1}{2}\,2^{i-2}
=\frac{23}{24}\,2^i
<2^i,\\
Y_i
&<
2^{i-1}+\frac{1}{2}\cdot\frac{2}{3}\,2^{i-1}
=\frac{2}{3}\,2^i.
\end{align*}

Note that for integers $i,j \ge 1$, we have
\begin{equation}\label{equ::ZiSubmultiplicative}
Z_{i+j} \le Z_iZ_j.
\end{equation}
Indeed, deleting the edge between the $i$-th and $(i+1)$-st vertices of $P_{i+j}$ gives the disjoint union of $P_i$ and $P_j$, and deleting any edge can only increase the weighted sum by the definition of $Z(G)$ and good functions.

Let $\gamma \ce \max_{8 \le j \le 15}Z_j^{1/j}$. Since $Z_j<2^j$ for every $j \ge 8$, we have $\gamma<2$. Every integer $i \ge 8$ is a sum of integers from $[8,15]$. Applying \cref{equ::ZiSubmultiplicative} repeatedly gives $Z_i \le \gamma^i$.
\end{proof}

\begin{claim}\label{cla::gktle}
There are absolute constants $K_1,\eta_1>0$ such that for all integers $k,t$ with $t \ge 0$ and $1 \le k \le t/4$,
\[
g_k(t) \le K_1(2-\eta_1)^t.
\]
\end{claim}

\begin{proof}
By definition, $Z(G_k(t))$ factors over the connected components of $G_k(t)$. Let $H$ be one such component. If $H$ contains no loop, then \cref{cla::gammaDecomposition,cla::ZvaluePathge8} give $Z(H) \le \gamma^{|V(H)|}$.

Suppose that $H$ contains a loop. Deleting all loops from $H$ can only increase $Z(H)$, and the resulting graph is a path. Therefore, if $|V(H)| \ge 8$, we again have $Z(H) \le \gamma^{|V(H)|}$.

By \cref{cla::gammaDecomposition}, at most two vertices have a loop, and hence at most two components contain a loop. Those having fewer than eight vertices contain at most fourteen vertices altogether. If all constraints on these vertices are discarded, each vertex has total contribution $1+3\cdot\frac{1}{2}=\frac{5}{2}$. Hence, the total contribution of all such exceptional components is at most $(5/2)^{14}$. Therefore,
\[
Z(G_k(t)) \le
\left(\frac{5}{2}\right)^{14}\gamma^t.
\]
Together with \cref{equ::gktleZGkt}, this gives
\[
g_k(t) \le
\frac{1}{2}\left(\frac{5}{2}\right)^{14}\gamma^t.
\]
Taking $K_1 \ce \frac{1}{2}(5/2)^{14}$ and $\eta_1 \ce 2-\gamma>0$ proves the claim.
\end{proof}

Note that $|J_t| \le t/4$, which is polynomial in $t$. \cref{lem::Ghosal} follows from \cref{equ::gtlesumg0tgkt} and \cref{cla::g0tle,cla::gktle} immediately.
\end{proof}

\end{document}